\documentclass[11pt,twoside]{amsart}
\usepackage{amsxtra}
\usepackage{color}
\usepackage{amsopn}
\usepackage{amsmath,amsthm,amssymb,arydshln}
\usepackage{mathrsfs,mathtools}
\usepackage{stmaryrd}
\usepackage{hyperref}
\usepackage{enumerate}
\usepackage{xcolor}
\usepackage{pifont}
\usepackage{adjustbox}

\newtheorem{theorem}{Theorem}[section]
\newtheorem{corollary}[theorem]{Corollary}
\newtheorem{proposition}[theorem]{Proposition}
\newtheorem{lemma}[theorem]{Lemma}
\newtheorem*{theorem*}{Theorem}
\theoremstyle{definition}
\newtheorem{definition}[theorem]{Definition}
\newtheorem*{notation}{Notation}
\newtheorem{example}[theorem]{Example}
\newtheorem{remark}[theorem]{Remark}

\newcommand{\R}{{\mathbb R}}

\newcommand{\C}{{\mathbb C}}
\newcommand{\Z}{\mathbb Z}
\newcommand{\T}{{\mathbb T}}

\newcommand{\W}{\wedge}

\newcommand{\f}{\varphi}

\newcommand{\psip}{\psi_{\scriptscriptstyle +}}
\newcommand{\psim}{\psi_{\scriptscriptstyle -}}
\newcommand{\Wd}{w_2^{\scriptscriptstyle -}}
\newcommand{\sh}{{\sharp_h}}

\newcommand{\SU}{{\rm SU}}

\newcommand{\GL}{{\rm GL}}

\newcommand{\G}{{\rm G}}

\newcommand{\SL}{{\rm SL}}
\newcommand{\Gd}{{\rm G}_2}
\newcommand{\ddt}{\frac{\partial}{\partial t}}
\newcommand{\dt}{\frac{\mathrm d}{{\mathrm d}t}}

\newcommand{\Ric}{{\rm Ric}}

\newcommand{\Hess}{{\rm Hess}}
\newcommand{\grad}{{\rm grad}}

\newcommand{\frg}{\mathfrak{g}}

\newcommand{\frn}{\mathfrak{n}}
\newcommand{\fre}{\mathfrak{e}}

\newcommand{\frsu}{\mathfrak{su}}

\newcommand{\+}{{{\scriptscriptstyle +}}}
\newcommand{\Su}{{\mathbb S}^1}
\newcommand{\st}{\ |\ }

\newcommand{\diag}{{\rm diag}}

\numberwithin{equation}{section}

\title[The Laplacian flow for closed warped G$_2$-structures]{Closed warped G$_{\mathbf2}$-structures evolving under the Laplacian flow}
\author{Anna Fino and Alberto Raffero}
\subjclass[2010]{53C44, 53C10, 53C30}
\keywords{Laplacian flow, $\G_2$-structure, $\SU(3)$-structure, warped product, soliton}
\thanks{The authors were supported by GNSAGA of INdAM. The first author was also supported by PRIN 2015 
``Real and complex manifolds: geometry, topology and harmonic analysis'' of MIUR}
\address{(Anna Fino) Dipartimento di Matematica ``G. Peano'' \\ Universit\`a degli Studi di Torino\\
Via Carlo Alberto 10\\
10123 Torino\\ Italy}
\email{annamaria.fino@unito.it}
\address{(Alberto Raffero) Dipartimento di Matematica e Informatica ``U.~Dini'' \\ Universit\`a degli Studi di Firenze\\ Viale Morgagni 67/a\\ 50134 Firenze\\ Italy}
\email{alberto.raffero@unifi.it}

\begin{document}

\begin{abstract}
We study the behaviour of the Laplacian flow evolving closed $\G_2$-structures on warped products of the form $M^6\times\Su$, 
where the base $M^6$ is a compact 6-manifold endowed with an $\SU(3)$-structure. 
In the general case, we reinterpret the flow as a set of evolution equations on $M^6$ for the differential forms defining the $\SU(3)$-structure and the warping function.  
When the latter is constant, we find sufficient conditions for the existence of solutions of the corresponding coupled flow. 
This provides a method to construct immortal solutions of the Laplacian flow on the product manifolds $M^6\times\Su$. 
The application of our results to explicit cases allows us to obtain new examples of expanding Laplacian solitons.
\end{abstract}

\maketitle

\section{Introduction}
Let $(M^m,g_m)$ and $(N^n,g_n)$ be two compact, connected Riemannian manifolds. 
Assume that $g_n$ is Ricci-flat, so that $g_n(t) = g_n$ is a stationary solution of the Ricci flow on $N^n.$
Consider a positive function $f\in\mathcal{C}^\infty(M^m)$ and the product manifold $M^m\times N^n$ endowed with the Riemannian metric $g \coloneqq g_m + f^2\,g_n.$ 
Then, $(M^m\times N^n,g)$ is a warped product, and the Ricci flow on it starting from $g(0) = g$ can be considered as 
a coupled flow for the metric $g_m$ and the warping function $f$ on the base manifold $M^m$:
\begin{equation}\label{RFWPeqns}
\renewcommand\arraystretch{1.4}
\left\{
\begin{array}{rcl}
\ddt g_m(t) 	&=&  -2\,\Ric(g_m(t)) +2\,\frac{n}{f(t)}\,\Hess_m(f(t)),\\
\ddt f(t) 		&=& \Delta_m f(t) + (n-1)\frac{1}{f(t)}|\nabla f(t)|_{g(t)}^2,\\
g_m(0)		&=& g_m,\\
f(0)			&=& f,
\end{array}
\right.
\end{equation}
where $\Delta_m = {\rm div}_m\circ\grad_m$ and $\Hess_m$ are the Laplacian and the Hessian on $(M^m,g_m(t))$, respectively (see for instance \cite{LotSes, Tran}). 
In particular, the Ricci flow preserves the warped product structure in the following sense: if $(g_m(t), f(t))$ is the solution of \eqref{RFWPeqns} starting from $(g_m,f)$ at $t=0$, then 
$g(t)=g_m(t) + [f(t)]^2g_n$ is the solution of the Ricci flow on $M^m \times N^n$ with initial condition $g(0) = g_m+f^2g_n$.  

In the study of special geometric structures on Riemannian manifolds,  
warped products are  a useful tool to construct new examples starting from known ones. 
Consider for instance a 6-manifold $M^6$ endowed with an $\SU(3)$-structure, namely a pair of differential forms $\omega\in\Omega^2(M^6)$, $\psip\in\Omega^3(M^6)$ 
which are stable in the sense of \cite{Hit} and satisfy certain compatibility conditions. 
Then, there are various methods to define a $\G_2$-structure on the product $M^6\times\Su$ by means of the $\SU(3)$-structure on $M^6$, 
and in each case this 7-manifold turns out to be a Riemannian product or a warped product (see e.g.~\cite{ChSa,CleIva,KarMcTsu}). 

Recall that a $\G_2$-structure on a 7-manifold $M^7$ is defined by a stable 3-form $\f$ giving rise to a Riemannian metric $g_\f$ and to a volume form $dV_\f$.  
The intrinsic torsion of a $\G_2$-structure $\f$ is completely determined by $d\f$ and $d*_\f\f$, $*_\f$ being the Hodge operator defined by $g_\f$ and $dV_\f$, 
and it vanishes identically if and only if both $\f$ and $*_\f\f$ are closed \cite{Bry,FeGr}. 
When this happens,  the Riemannian holonomy group ${\rm Hol}(g_\f)$ is a subgroup of $\G_2$, and $g_\f$ is Ricci-flat. 
$\G_2$-structures whose defining 3-form is both closed and co-closed are called {\em torsion-free}, 
and they play a central role in the construction of metrics with holonomy $\G_2$. 
The first complete examples of such metrics were obtained by Bryant and Salamon in \cite{BrySal}, 
while compact examples of Riemannian manifolds with holonomy $\G_2$ were constructed first by Joyce \cite{Joy1}, and then by Kovalev \cite{Kov}, 
and by Corti, Haskins, Nordstr\"om, Pacini \cite{CHNP}.
A potential method to obtain new results in this direction is represented by geometric flows evolving $\G_2$-structures. 

Let $M^7$ be a 7-manifold endowed with a {\em closed} $\G_2$-structure $\f$, i.e., satisfying $d\f=0$. 
The {\em Laplacian flow} starting from $\f$ is the initial value problem 
\[
\begin{cases}
\ddt \f(t) = \Delta_{\f(t)}\f(t),\\
d \f(t)=0,\\
\f(0)=\f,
\end{cases}
\]
where $\Delta_{\f(t)}$ denotes the Hodge Laplacian of the Riemannian metric $g_{\f(t)}$ induced by $\f(t)$. 
This geometric flow was introduced by Bryant in \cite{Bry}  as a tool to find torsion-free $\G_2$-structures on compact manifolds. 
Short-time existence and uniqueness of the solution when $M^7$ is compact were proved by Bryant and Xu in the unpublished paper \cite{BryXu}.  
Recently, Lotay and Wei investigated the properties of the Laplacian flow in the series of papers \cite{LotWei1,LotWei2,LotWei3}. 

In \cite{KarMcTsu}, a flow evolving the 4-form $*_\f\f$ in the direction of minus its Hodge Laplacian was introduced,    
and its behaviour on warped products with base manifold $\Su$ and fibre a 6-manifold endowed with a nearly K\"ahler or a torsion-free $\SU(3)$-structure was investigated. 

In the recent paper \cite{FiYa}, Fine and Yao studied the evolution of a hypersymplectic structure $(\omega_1,\omega_2,\omega_3)$ on a compact 4-manifold $M^4$, 
assuming that the closed $\G_2$-structure induced by it on the 7-manifold $M^4\times\T^3$ evolves under the Laplacian flow.

Let $M^6$ be a compact 6-manifold endowed with an $\SU(3)$-structure $(\omega,\psip)$, and denote by $h$ the associated Riemannian metric.   
Motivated by the results recalled so far, in this paper we are interested in studying the behaviour of the Laplacian flow on warped products of the form 
$(M^6\times\Su,\f)$, where the fibre is the circle $\Su$, and the $\G_2$-structure is  defined by
\begin{equation}\label{g2strwp}
\f = f\,\omega\W ds+\psip,
\end{equation}
with $f\in\mathcal{C}^\infty(M^6)$ a positive function, and $s$ the angle coordinate on $\Su$. 
We call a $\G_2$-structure given by \eqref{g2strwp} {\em warped}, as it induces the warped product metric $g_\f=h+f^2 ds^2$.

Since $M^6\times\Su$ is compact,  the Laplacian flow starting from a closed warped $\G_2$-structure has a unique solution defined for short times. 
It is then natural to investigate whether there exists a relation between the Laplacian flow on $M^6\times\Su$ and a coupled flow for the $\SU(3)$-structure $(\omega,\psip)$ 
and the warping function $f$ on $M^6,$ and to study under which conditions the Laplacian flow preserves the expression \eqref{g2strwp} of $\f$.

The present paper begins with a review of basic facts about  $\SU(3)$- and $\G_2$-structures in Section \ref{settingsect}. 
In Section \ref{warpedclosedg2sect}, we study the properties of warped $\G_2$-structures. 
In particular, we observe that a warped $\G_2$-structure is closed if and only if the underlying $\SU(3)$-structure 
satisfies the equations $d\psip=0$ and $d\omega=\theta\W\omega$, where $\theta=-d\log(f)$.  
Then, we provide a method to construct such type of $\SU(3)$-structures, and we use it to obtain an explicit example on the 6-torus. 
Section \ref{lapflowgensect} is devoted to investigate the equivalence between the Laplacian flow evolving closed warped $\G_2$-structures and a coupled flow for 
the $\SU(3)$-structure and the warping function $f$ on $M^6$ (Proposition \ref{summinducedflow}). 
In Section \ref{lapflowproduct}, we restrict our attention to Riemannian product manifolds by assuming that the function $f$ is constant. 
In this case, the $\G_2$-structure given by \eqref{g2strwp} on $M^6\times\Su$ is closed if and only if the $\SU(3)$-structure is {\em symplectic half-flat}, namely if and only 
if both $\omega$ and $\psip$ are closed. Under this hypothesis, the coupled flow obtained in Section \ref{lapflowgensect} has a simpler expression, 
and we are able to get sufficient conditions on the initial datum $(\omega,\psip)$ guaranteeing the (long-time) existence of solutions (Theorem \ref{thmlapshf}). 
This provides a method to construct solutions of the Laplacian flow on $M^6\times\Su$ defined on a maximal interval of time $(T,+\infty)$, with $T<0$. 
In the last section, we apply our results to explicit examples. 
Starting with the classification of unimodular solvable Lie groups $\G$ admitting a left-invariant symplectic half-flat $\SU(3)$-structure \cite{FMOU}, 
we show that each one of them admits an $\SU(3)$-structure satisfying the hypothesis of Theorem \ref{thmlapshf}.  
It turns out that in some cases the corresponding closed $\G_2$-structure $\f$ is a {\em Laplacian soliton}, i.e., 
it satisfies the equation $\Delta_\f\f=\mathcal{L}_X\f+\lambda\f$, for some vector field $X\in\mathfrak{X}(\G\times\R)$ and some real number $\lambda$. 
This gives new examples of Laplacian solitons on solvable Lie groups in addition to those obtained in \cite{FFM,Lau1,Lau2,Lau3,Nic}.  
We conclude the section describing the solution of the Laplacian flow on the product of a twistor space endowed with a symplectic half-flat $\SU(3)$-structure and the circle, 
and giving two examples of non-flat closed $\G_2$-structures on the 7-torus.

\section{Preliminaries on SU(3)- and G$_2$-structures}\label{settingsect}
A $\G$-structure on an $m$-dimensional smooth manifold $M^m$ is a reduction of the structure group of the frame bundle from $\GL(m,\R)$ to a subgroup $\G$. 
It is known, see e.g.~\cite{CLSS, Hit1, Hit, Rei}, that when $\G=\SU(3)$ and $\G=\G_2$, the existence of a $\G$-structure is equivalent to the existence 
of suitable differential forms on $M^m$ satisfying the following non-degeneracy condition: 
at each point $p$ of $M^m$, their orbit under the natural action of the group $\GL(T_pM^m)$ is open. 
Such forms are usually called {\em stable} in literature.  
It is clear that being stable is an open property. Thus, small perturbations of a stable form are stable, too. 
For the sake of convenience, we briefly recall the algebraic description of those stable forms which are of interest for us in the next propositions.

\begin{proposition}[\cite{CLSS,Hit1,Hit}]
Let $V$ be an oriented, six-dimensional real vector space. Then
\begin{enumerate}[-]
\item a 2-form $\omega\in\Lambda^2(V^*)$ is stable if and only if it is non-degenerate, i.e., $\omega^3\neq0$;
\item every 3-form $\psi\in\Lambda^3(V^*)$ gives rise to an irreducible polynomial $P(\psi)$ of degree 4 which vanishes identically if and only if $\psi$ is not stable. 
\end{enumerate}
$P(\psi)$ is defined as follows. 
Fix a volume form $\nu\in\Lambda^6(V^*)$ on $V,$ consider the isomorphism $A:\Lambda^5(V^*)\rightarrow V\otimes\Lambda^6(V^*)$ induced by the wedge product 
$\W:\Lambda^5(V^*)\otimes V^*\rightarrow\Lambda^6(V^*)$, and let $K_\psi:V\rightarrow V$ be the linear map defined via the identity $K_\psi(v)\otimes\nu = A(\iota_v\psi\W\psi)$, 
$\iota_v\psi$ being the contraction of $\psi$ by the vector $v\in V.$ Then, $P(\psi) \coloneqq \frac16{\rm tr}(K_\psi^2)$. 
\end{proposition}

We denote by $\Lambda^3_\+(V^*)$ the open orbit of stable 3-forms satisfying $P(\psi)<0$. 
The $\GL^\+(V)$-stabilizer of a 3-form lying in this orbit is isomorphic to $\SL(3,\C)$. 
As a consequence, every $\psi\in\Lambda^3_{\+}(V^*)$ gives rise to a complex structure  
\[
J_\psi:V\rightarrow V,\quad J_\psi \coloneqq \frac{1}{\sqrt{|P(\psi)|}}\,K_\psi,
\]   
which depends only on $\psi$ and on the volume form $\nu$.  
Moreover, the complex form $\psi+i J_\psi\psi$ is of type $(3,0)$ with respect to $J_\psi$, and the real 3-form $J_\psi\psi$ is stable, too. 

\begin{proposition}[\cite{CLSS,Hit}]
Let $W$ be an oriented, seven-dimensional real vector space. 
Consider a 3-form $\f\in\Lambda^3(W^*)$ and the symmetric bilinear map
\[
b_\f:W\times W\rightarrow\Lambda^7(W^*),\quad b_\f(v,w)=\frac16\,\iota_v\f\W\iota_w\f\W\f.
\]
Then, $\f$ is stable if and only if $\det(b_\f)^{\frac19}\in\Lambda^7(W^*)$ defines a volume form on $W.$
\end{proposition}

We use the notation $\Lambda^3_\+(W^*)$ to indicate  the open orbit of stable 3-forms for which the symmetric bilinear map 
$g_\f \coloneqq \det(b_\f)^{-\frac19}b_\f:W\times W\rightarrow\R$ is positive definite. 
The $\GL^\+(W)$-stabilizer of a 3-form belonging to this orbit is isomorphic to the exceptional Lie group $\G_2$. 

\begin{notation}
In what follows, we denote by $\Omega^k_\+(M^m)$ the space of sections of the open subbundle of $\Lambda^k(T^*M^m)$ whose fibre over each point $p\in M^m$ is given by
 $\Lambda^k_\+(T^*_pM^m)$.
\end{notation}

\subsection{SU(3)-structures on six-dimensional manifolds}
Let $M^6$ be a six-dimensional manifold, and let $\omega\in\Omega^2(M^6)$, $\psip\in\Omega^3_\+(M^6)$ be a pair of stable forms satisfying the compatibility condition 
\begin{equation}\label{compcond}
\omega\W\psip=0.
\end{equation}
Consider the almost complex structure $J=J_{\psip}$ determined by $\psip$ and the volume form $\frac{\omega^3}{6}$. 
Then, the 3-form $\psip$ is the real part of a nowhere vanishing $(3,0)$-form $\Psi \coloneqq \psip+i\psim$  
with $\psim \coloneqq J\psip = \psip(J\cdot,J\cdot,J\cdot) =  -\psip(J\cdot,\cdot,\cdot)$,  
where the last identity holds since $\psip$ is of type $(3,0)+(0,3)$ with respect to $J$. Moreover, 
$\omega$ is of type $(1,1)$ and, as a consequence, the 2-covariant tensor $h(\cdot,\cdot) \coloneqq \omega(\cdot,J\cdot)$ is symmetric. 
Under these assumptions, the pair $(\omega,\psip)$ defines an $\SU(3)$-structure on $M^6$ provided that $h$ is a Riemannian metric and the following 
normalization condition is satisfied
\begin{equation}\label{normcond}
\psip\W\psim=\frac23\,\omega^3.
\end{equation}
Conversely, every $\SU(3)$-structure on a 6-manifold arises in this fashion (cf.~\cite{Hit}). 

We denote by $|\cdot|_h$ the pointwise norm induced by $h$, by $dV_h=\frac{\omega^3}{6}$ the Riemannian volume form of $h$, 
and by $*_h$ the Hodge operator defined by $h$ and $dV_h$. The Hodge duals of $\omega$ and $\psip$ are given by
\begin{equation}\label{starSU3}
*_h\omega=\frac12\,\omega^2,\quad *_h\,\psip=\psim.
\end{equation}

By \cite[Thm.~1.1]{ChSa}, the intrinsic torsion of an $\SU(3)$-structure is determined by $d\omega$, $d\psip$, and $d\psim$. 
More precisely, the $\SU(3)$-irreducible decompositions of the modules $\Lambda^k\left((\R^6)^*\right)$, $k=3,4$, induce on $M^6$ the $h$-orthogonal splittings 
\begin{equation}\label{3formdec6}
\Omega^3(M^6) = \mathcal{C}^\infty(M^6)\,\psip \oplus \mathcal{C}^\infty(M^6)\,\psim \oplus \Omega^3_{12}(M^6) \oplus \Omega^1(M^6)\W\omega,
\end{equation}
\begin{equation}\label{4formsdec6}
\Omega^4(M^6) = \mathcal{C}^\infty(M^6)\,\omega^2 \oplus \Omega^2_8(M^6)\W\omega \oplus\Omega^1(M^6)\W\psip,
\end{equation}
where, following the notation of \cite{BedVez}, 
\begin{eqnarray*}
\Omega^2_8(M^6)		&\coloneqq&	\left\{\beta\in\Omega^2(M^6)\st \beta\W\omega^2=0,~J\beta=\beta\right\},\\
\Omega^3_{12}(M^6)	&\coloneqq&	\left\{\beta\in\Omega^3(M^6) \st \beta\W\omega=0,~\beta\W\psi_{\scriptscriptstyle\pm}=0 \right\}.
\end{eqnarray*}
The differential forms $d\omega$, $d\psip$, $d\psim$ decompose accordingly, and each of their summands corresponds to a component of the intrinsic torsion in the 
$\SU(3)$-irreducible splitting of $(\R^6)^*\otimes\frsu(3)^\perp$. 
In particular, the intrinsic torsion vanishes identically if and only if $\omega$, $\psip$, and $\psim$ are all closed. 
Moreover, $\SU(3)$-structures can be divided into classes according to the vanishing of the components of $d\omega$, $d\psip$, $d\psim$. 

We shall review the properties of those $\SU(3)$-structures which are of interest for us in due course, while we refer the reader to \cite{BedVez,ChSa} for a more detailed description. 

\subsection{G$_{\mathbf 2}$-structures on seven-dimensional manifolds} 
A $\G_2$-structure on a 7-manifold $M^7$ is characterized by the existence of a stable 3-form $\f\in\Omega^3_\+(M^7)$ 
giving rise to a Riemannian metric $g_\f$ and to a volume form $dV_\f$ on $M^7$  via the identity  
\[
g_\f(X,Y)\,dV_\f = \frac16\,\iota_X\f\W\iota_Y\f\W\f, \quad X,Y\in\mathfrak{X}(M^7).
\]
We denote by $*_\f$ the Hodge operator defined by $g_\f$ and $dV_\f$, and by $|\cdot|_\f$ the pointwise norm induced by $g_\f$. 

Consider the Levi Civita connection $\nabla^\f$ of $g_\f$. By \cite{Bry,FeGr}, the intrinsic torsion of the $\G_2$-structure $\f$ can be identified with $\nabla^\f\f$, 
and it is completely determined by the exterior derivatives $d\f$ and $d*_\f\f$.
If both $\f$ and $*_\f\f$ are closed, the intrinsic torsion vanishes identically, the Riemannian holonomy group ${\rm Hol}(g_\f)$ is a subgroup of $\G_2$, and $g_\f$ is Ricci-flat. 
When this happens, the $\G_2$-structure is said to be {\em torsion-free}.

In this paper, we focus only on the class of {\em closed} $\G_2$-structures $\f$, i.e., satisfying $d\f=0$.  
In this case, there exists a unique 2-form $\tau\in \Omega^2_{14}(M) \coloneqq \left\{\beta\in\Omega^2(M) \st \beta\W*_{\f}\f = 0\right\}$  
such that $d*_\f\f=\tau\W\f$ (cf.~\cite[Prop.~1]{Bry}). The 2-form $\tau$ is known as  {\em intrinsic torsion form} of $\f$. 
Examples of closed $\G_2$-structures were obtained for instance in \cite{ClSw,CoFe,Fer}.

\section{Closed warped G$_2$-structures}\label{warpedclosedg2sect}
Let $M^6$ be a connected six-dimensional manifold endowed with an $\SU(3)$-structure $(\omega,\psip)$. 
\begin{definition} 
A $\G_2$-structure on the 7-manifold $M^7\coloneqq M^6\times\Su$ is {\em warped} if it is of the form $\f=f\,\omega\W ds+\psip$, 
for some positive function $f\in\mathcal{C}^\infty(M^6)$.
\end{definition}

The Riemannian metric induced by a warped $\G_2$-structure $\f$ is $g_\f=h+f^2ds^2$, and the associated volume form is $dV_\f=f\,dV_h\W ds$. 
Therefore, $(M^7,g_\f)$ is a warped product with warping function $f$, base manifold $M^6$ and fibre $\Su$. 

The Hodge dual of the 3-form $\varphi$  defined by \eqref{g2strwp} is
\begin{equation}\label{g2strwpstar}
*_\f\f = \frac12\,\omega^2 + f\,\psim\W ds. 
\end{equation}
This can be checked using \eqref{starSU3} and the next result, whose proof is an immediate consequence of the definition of the Hodge operator. 
\begin{lemma}\label{HodgestarWP}
Let $\beta\in\Omega^k(M^6)$ be a  differential $k$-form on $M^6,$ and let $*_h$ and $*_\f$ be the Hodge operators determined by $(h, dV_h)$ and $(g_\f, dV_\f)$, respectively. Then
\begin{enumerate}[{\rm i)}]
\item\label{warpi} $*_\f\beta = f*_h\beta\W ds$;
\item\label{warpii} $*_\f(\beta\W ds) = (-1)^kf^{-1}*_h\beta$.
\end{enumerate}
\end{lemma}

\begin{remark}
It is worth observing here that the 3-form $\varphi$  given  by \eqref{g2strwp} defines a different type of warped $\G_2$-structure 
when the base manifold is $\Su$ and $f\in\mathcal{C}^\infty(\Su)$.   
Such kind of warped $\G_2$-structures were considered for instance in \cite{CleIva,KarMcTsu}.  
\end{remark}

\begin{notation}
For the sake of clarity, from now on we shall denote the exterior derivative and its formal adjoint on an $m$-dimensional oriented Riemannian manifold $(M^m,g,dV)$ 
by $d_m$ and $d^*_m$, respectively. 
Recall that for any $k$-form $\beta\in\Omega^k(M^m)$ it holds $d^*_m\beta = (-1)^{m(k+1)+1}*d_m*\beta$, where $*$ is the Hodge operator determined by $(g,dV)$. 
\end{notation}

Let us focus on closed warped $\Gd$-structures. 
First of all, observe that it is possible to characterize them in terms of $d_6\omega$ and $d_6\psip$. 
\begin{lemma}\label{G2closedSU3corresp}
A warped $\G_2$-structure $\f=f\,\omega\W ds+\psip$ on $M^7 = M^6 \times  \Su$ is closed if and only if 
$(\omega,\psip)$ satisfies 
\begin{equation}\label{structueqnsu3}
\begin{cases}
d_6\omega = \theta\W\omega,\\
d_6\psip = 0,\\
d_6\psim = \Wd\W\omega,
\end{cases}
\end{equation}
where $\theta \coloneqq -d_6\log(f)$ and $\Wd\in\Omega^2_8(M^6)$ are the intrinsic torsion forms of $(\omega,\psip)$.
\end{lemma}
\begin{proof}
The assertion follows from the identity $d_7\f=(d_6f\W\omega+f\,d_6\omega)\W ds+d_6\psip$, and the results of \cite[Sect.~2.5]{BedVez}. 
Notice that the third equation in \eqref{structueqnsu3} is automatically implied by the first two. 
\end{proof}

\begin{remark}\label{W2W4remark}\ 
\begin{enumerate}[i)]
\item Following the notation of \cite{ChSa}, the intrinsic torsion of an $\SU(3)$-structure satisfying  \eqref{structueqnsu3} belongs to $\mathcal{W}_2^-\oplus\mathcal{W}_4$. 
\item\label{W2W4remark2} Since $\theta=-d_6\log(f)$ is an exact 1-form, the first equation in \eqref{structueqnsu3} implies that $\omega$ is {\it globally conformal symplectic}. 
More precisely, the global conformal change $f\omega$ defines a symplectic form on $M^6$. 
Furthermore, the pair 
\[
\widetilde\omega \coloneqq f\omega,\quad \widetilde{\psi}_{\scriptscriptstyle+} \coloneqq f^{\frac32}\psip,
\]
defines an $\SU(3)$-structure such that $d_6\widetilde\omega=0$ and whose intrinsic torsion belongs to $\mathcal{W}_2^-\oplus\mathcal{W}_5$.  
\item\label{Leef} $\theta$ is the {\em Lee form} of the almost Hermitian structure $(\omega,J)$, thus $d_6^*\omega = 2 J\theta$.
\item\label{rempropWd} $\Wd$ is a primitive 2-form of type $(1,1)$. Hence, it satisfies the identities $\Wd\W\psi_{\scriptscriptstyle \pm}=0$, $\Wd\W\omega^2=0$, and 
$\Wd\W\omega=-*_h\Wd$. 
Consequently, $\Wd=d_6^*\psip$ when \eqref{structueqnsu3} holds. 
\item\label{fconstiffshf} When $f$ is constant, the 1-form $\theta$ vanishes identically and both $\omega$ and $\psip$ are closed. 
In such a case, $(\omega,\psip)$ is a {\it symplectic half-flat} $\SU(3)$-structure (\cite{ConTom,DeBTom,TomVez}). 
The converse is also true: if $\f=f\,\omega\W ds+\psip$ is closed and $(\omega,\psip)$ is symplectic half-flat, then $d_6f\W\omega=0$, which implies that $d_6f=0$. 
\end{enumerate}
\end{remark}

We describe now a possible strategy to construct examples of $\SU(3)$-structures satisfying \eqref{structueqnsu3}.   
Assume that $M^6$ is endowed with a symplectic 2-form $\omega$ and a closed stable 3-form $\rho\in\Omega^3_\+(M^6)$ inducing an almost complex structure $J$ 
which is {\em calibrated} by $\omega$, 
i.e., $\omega$ is of type $(1,1)$ with respect to $J$ and the symmetric tensor $\omega(\cdot,J\cdot)$ defines a Riemannian metric.  
Then, $\rho+iJ\rho$ is a nowhere vanishing complex $(3,0)$-form. In general, the normalization condition \eqref{normcond} is not satisfied by the pair $(\omega,\rho)$. 
Nevertheless, there exists a positive function $u\in\mathcal{C}^\infty(M^6)$ such that
\[
u^3\rho\W J\rho = \frac23\,\omega^3. 
\]
The desired $\SU(3)$-structure is then given by $\left(\widehat\omega\coloneqq{u}^{\scriptscriptstyle-1}\omega, \rho\right)$. 

\begin{example}\label{extorus}
In \cite{TomVez}, an example of a family of symplectic half-flat $\SU(3)$-structures on the 6-torus $\T^6 = \R^6/\Z^6$ was given. Let us recall it.  
Denote by $(x^1,\ldots,x^6)$ the standard coordinates on $\R^6$, let $a(x^1)$, $b(x^2)$ and $c(x^3)$ be three $\Z^6$-periodic and nonconstant 
smooth functions on $\R^6$, and let 
\[
\ell_1 \coloneqq b(x^2)-c(x^3),\quad \ell_2\coloneqq c(x^3)-a(x^1),\quad \ell_3\coloneqq a(x^1)-b(x^2).
\]
Then, for each $r>0$ the symplectic half-flat $\SU(3)$-structure on $\T^6$ is given by the following pair of $\Z^6$-invariant differential forms
\begin{equation}\label{SHFT6}
\renewcommand\arraystretch{1.4}
\begin{array}{rcl}
\omega &=& dx^{14}+dx^{25}+dx^{36},\\
\psip(r) &=& -e^{r\,\ell_3}\,dx^{126} +e^{r\,\ell_2}\,dx^{135}  -e^{r\,\ell_1}\,dx^{234} +dx^{456},
\end{array}
\end{equation}
where $dx^{ijk\cdots}$ is a shorthand for the wedge product $dx^i\W dx^j \W dx^k \W \cdots$. 
The almost complex structure $J_r$ induced by $(\omega,\psip(r))$ is the following 
\[
J_r\left(\frac{\partial}{\partial x^k} \right) = e^{-r\,\ell_k}\frac{\partial}{\partial x^{k+3}},\quad J_r\left(\frac{\partial}{\partial x^{k+3}} \right) = -e^{r\,\ell_k}\frac{\partial}{\partial x^k},\quad 
k=1,2,3.
\]

Consider three positive, $\Z^6$-periodic smooth functions $\kappa_1(x^1,x^4)$, $\kappa_2(x^2,x^5)$, $\kappa_3(x^3,x^6)$, and the closed, $\Z^6$-invariant 2-form
\begin{equation}\label{tildeomgT6}
\widetilde{\omega}\coloneqq (\kappa_1)^3\,dx^{14}+(\kappa_2)^3\,dx^{25}+(\kappa_3)^3\,dx^{36}.
\end{equation}
$\widetilde\omega$ is a symplectic form on $\T^6$, it is of type $(1,1)$ with respect to the almost complex structure $J_r$,  
since $\widetilde\omega(J_r\cdot,J_r\cdot)=\widetilde\omega$, and the symmetric tensor $\widetilde\omega(\cdot,J_r\cdot)$ is given by
\[
{e^{-r\,\ell_1}\kappa_1^3}(dx^1)^2 + {e^{-r\,\ell_2}\kappa_2^3}(dx^2)^2 + {e^{-r\,\ell_3}\kappa_3^3}(dx^3)^2 +
{e^{r\,\ell_1}\kappa_1^3}(dx^4)^2 + {e^{r\,\ell_2}\kappa_2^3}(dx^5)^2 + {e^{r\,\ell_3}\kappa_3^3}(dx^6)^2.
\]
Hence, $J_r$ is $\widetilde\omega$-calibrated. Moreover,
\[
\frac23\,\widetilde\omega^3 = \frac23\,(\kappa_1\kappa_2\kappa_3)^3\,\omega^3 = (\kappa_1\kappa_2\kappa_3)^3\, \psip(r)\W\psim(r). 
\]
Therefore, the pair $\left(\widehat\omega\coloneqq \frac{1}{\kappa_1\kappa_2\kappa_3}\,\widetilde\omega,~\psip(r)\right)$ defines a family of 
$\SU(3)$-structures on $\T^6$ satisfying \eqref{structueqnsu3} with $\theta =  -d_6 \log\left(\kappa_1\kappa_2\kappa_3\right)$. 
\end{example}

The intrinsic torsion forms $\theta$, $\Wd$ of an $\SU(3)$-structure fulfilling \eqref{structueqnsu3} completely determine the intrinsic torsion form 
$\tau$ of the corresponding closed warped $\G_2$-structure. 
In the next lemma, we summarize some identities which are useful to show this assertion.  
\begin{lemma}\label{su3identities}
Let $(\omega,\psip)$ be an $\SU(3)$-structure on a 6-manifold $M^6,$ and consider a 1-form $\beta\in\Omega^1(M^6)$. 
Then 
\begin{equation}\label{staralphapsim}
*_h(\beta\W\psim) = \iota_{\beta^\sh}\psip, 
\end{equation}
where $\beta^\sh$ is the $h$-dual vector field of $\beta$.  
As a consequence, the following identities hold
\begin{enumerate}[{\rm i)}]
\item\label{idi} $*_h(\beta\W\psim)\W\omega =J\beta\W\psip = \beta\W\psim$;
\item\label{idii} $*_h(\beta\W\psim)\W\omega^2 = 0$;
\item\label{idiii} $*_h(\beta\W\psim)\W\psip = -*_h(\beta\W\psip)\W\psim =  \beta\W\omega^2=2*_h(J\beta)$;
\item\label{idiv} $*_h(\beta\W\psim)\W\psim = *_h(\beta\W\psip)\W\psip  = -J\beta\W\omega^2=2*_h\beta$.
\end{enumerate}
\end{lemma}
\begin{proof}
\eqref{staralphapsim} is a consequence of the identity 
\begin{equation}\label{wedgestarcontraction}
*(\beta\W\sigma) = (-1)^k\,\iota_{\beta^\sharp}*\sigma,
\end{equation}
which holds for any 1-form $\beta$ and any $k$-form $\sigma$ on an oriented Riemannian manifold (see e.g.~\cite{Kar2}). 
To show \ref{idi}), it is sufficient to consider the contraction of $\omega\W\psip=0$ by $\beta^\sh$ and to notice that 
$\iota_{\beta^\sh}\omega=-J\beta$. 
\ref{idii}) follows immediately from \ref{idi}). 
\ref{idiii}) and \ref{idiv}) can be proved in a similar way. 
First, starting from the $(3,0)$-form $\Psi = \psip+i\psim$, the general equation $\iota_X\Psi\W\Psi=0$ gives
\[
\iota_X\psip\W\psip = \iota_X\psim\W\psim,\quad \iota_X\psip\W\psim = -\iota_X\psim\W\psip. 
\]
Using these equations and \eqref{wedgestarcontraction}, the first identities in \ref{idiii}) and in \ref{idiv}) follow.  
The remaining identities can be obtained contracting both members of the normalization condition \eqref{normcond} by $J\beta^\sh$  
and observing that $\iota_X\psip =\iota_{JX}\psim$ holds for any vector field $X\in\mathfrak{X}(M^6)$. 
\end{proof}

\begin{proposition}\label{tauwp}
The intrinsic torsion form $\tau\in\Omega^2_{14}(M^7)$ of a closed warped $\Gd$-structure $\f=f\,\omega\W ds +\psip$ on $M^7=M^6\times\Su$ has the following expression
\begin{equation}\label{tau2WP}
\tau 	= \Wd +*_h(\theta\W\psim)+2\,fJ\theta\W ds.
\end{equation}
\end{proposition}
\begin{proof}
By the uniqueness of $\tau$, it is sufficient to show that the 2-form \eqref{tau2WP} belongs to $\Omega^2_{14}(M^7)$ and satisfies $d_7*_\f\f=\tau\W\f$.  
Using the expression \eqref{g2strwpstar} of $*_\f\f$, the identities $\Wd\W\psim=0$ and $\Wd\W\omega^2=0$, and points \ref{idii}) and \ref{idiv})  
of Lemma \ref{su3identities}, we get
\begin{eqnarray*}
\tau\W*_\f\f 	&=& f*_h(\theta\W\psim)\W\psim\W ds + \frac12*_h(\theta\W\psim)\W\omega^2+fJ\theta\W ds\W \omega^2\\
			&=& -fJ\theta\W\omega^2\W ds + fJ\theta\W\omega^2\W ds \\
			&=&0.
\end{eqnarray*}
This proves that $\tau\in\Omega^2_{14}(M^7)$. From \eqref{structueqnsu3}, we have
\[
d_7*_\f\f = d_6f\W\psim\W ds +f\,\Wd\W\omega\W ds + \theta\W\omega^2.
\]
Now, from the identity $\Wd\W\psip=0$ and points \ref{idi}), \ref{idiii}) of Lemma \ref{su3identities}, we obtain
\begin{eqnarray*}
\tau\W\f 		&=& f\,\Wd\W\omega\W ds + f*_h(\theta\W\psim)\W\omega\W ds +*_h(\theta\W\psim)\W\psip +2\,fJ\theta\W ds\W\psip\\
			&=& f\,\Wd\W\omega\W ds + f\,\theta\W\psim\W ds + \theta\W\omega^2 -2\,f\,\theta\W\psim\W ds\\
			&=& d_7*_\f\f. 
\end{eqnarray*}
\end{proof}
 
Using point \ref{idiii}) of Lemma \ref{su3identities}, the identity \eqref{staralphapsim}, and $\theta\W\omega^2 = d*_h\omega$, 
we obtain the following equivalent expressions of $\tau$:
\[
\tau = \Wd +*_h(\theta\W\psim)-f*_h(\theta\W\omega^2)\W ds = \Wd +\iota_{\theta^\sh}\psip+f\,d^*_6\omega\W ds.
\]
 
By Lemma \ref{HodgestarWP}, the Hodge dual of $\tau$ is
\[
*_\f\tau = f*_h\Wd\W ds + f\,\theta\W\psim\W ds -\theta\W\omega^2.
\] 
As a consequence, the norm $|\tau|_\f^2 	= \left|\Wd\right|_h^2+\left|\theta\W\psim\right|_h^2+\left|\theta\W\omega^2\right|_h^2$ is a function on $M^6$.  
Notice that from point \ref{idiv}) of Lemma \ref{su3identities} we have
\[
\left|\theta\W\psim\right|_h^2\,dV_h = \theta\W\psim\W*_h(\theta\W\psim) = \theta\W(2*_h\theta) = 2\left|\theta\right|_h^2\,dV_h. 
\]
Similarly, using point \ref{idiii}) of the same lemma, we see that 
\begin{equation}\label{twoeqn}
\left|\theta\W\omega^2\right|_h^2=4\left|\theta\right|_h^2. 
\end{equation}
Hence, we get
\begin{equation}\label{normtau2}
|\tau|_\f^2 	= \left|\Wd\right|_h^2+6\left|\theta\right|_h^2.
\end{equation}

Let $\Delta_\f\coloneqq d_7 d_7 ^{*}+d_7^{*} d_7$ denote the Hodge Laplacian on $M^7$ with respect to the Riemannian metric $g_\f$ induced by $\f$. 
We can use Proposition \ref{tauwp} to compute the expression of $\Delta_\f\f$ for a closed warped $\Gd$-structure $\f$. 
First, notice that in such a case $\Delta_\f\f = -d*_\f d*_\f\f = d_7\tau$. Thus, we get
\begin{equation}\label{Lapphi}
\Delta_\f\f = d_6\Wd +d_6(\iota_{\theta^\sh}\psip) +d_6(fd_6^*\omega)\W ds. 
\end{equation}
In the next result, we prove  that some summands appearing in the right-hand side of the above identity can be obtained by applying suitable 
second order differential operators to $\omega$ and to $\psip$. 
In what follows,  $\Delta_h$ denotes the Hodge Laplacian of $h$ on $M^6$. Recall that for every $\beta\in\Omega^k(M^6)$ 
\[
\Delta_h\beta =(d_6d_6^*+d_6^*d_6)\beta = (-d_6*_hd_6*_h-*_hd_6*_hd_6)\beta.
\]
\begin{proposition}\label{propDeltaphiwarp}
Let $\f=f\,\omega\W ds+\psip$ be a closed  warped $\Gd$-structure on $M^6\times\Su,$ where $(\omega,\psip)$ is an $\SU(3)$-structure on $M^6$ fulfilling \eqref{structueqnsu3}. 
Then 
\begin{equation}\label{HLWP}
\Delta_\f\f = \Delta_h\psip +d_6\left(\iota_{\theta^\sh}\psip\right) +f\left(d_6d_6^*\omega+\frac12\,(Jd_6^*\omega)\W d_6^*\omega \right)\W ds.
\end{equation}
In particular, when the warping function $f$ is constant, we have
\[
\Delta_\f\f = \Delta_h\psip.
\]
\end{proposition}
\begin{proof}
Consider equation \eqref{Lapphi}.  From $d_6\psim=\Wd\W\omega$, and the identities $\psim=*_h\psip$, $*_h(\Wd\W\omega)=-\Wd$, we have 
$*_hd_6*_h\psip = -\Wd$. Thus, as $d_6\psip=0$, it holds
\[
d_6\Wd = \Delta_h\psip.
\]  
Observe now that $\theta$ is the Lee form of $(\omega,J)$. Hence,  $\theta= -\frac12\,J\,d_6^*\omega$, and we get
\[
d_6(fd_6^*\omega) = d_6f\W d_6^*\omega + fd_6d_6^*\omega = \frac12\,f\,(Jd_6^*\omega)\W d_6^*\omega + fd_6d_6^*\omega.
\]
This proves \eqref{HLWP}. The last assertion follows observing that $f$ is constant if and only if $\theta=0$ and that this happens if and only if $d_6^*\omega=0$. 
\end{proof}

\section{Laplacian flow for closed warped $\G_2$-structures}\label{lapflowgensect}
\subsection{Laplacian flow and Laplacian solitons}
Let $M^7$ be a compact 7-manifold endowed with a closed $\G_2$-structure $\f$. 
The {\em Laplacian flow} starting from $\f$ is the initial value problem 
\begin{equation}\label{LapFlow}
\begin{cases}
\ddt \f(t) = \Delta_{\f(t)}\f(t),\\
d_7 \f(t)=0,\\
\f(0)=\f,
\end{cases}
\end{equation}
where $\Delta_{\f(t)}$ denotes the Hodge Laplacian of the Riemannian metric $g_{\f(t)}$ induced by $\f(t)$. 
In \cite[Thm.~0.1]{BryXu}, it was proved that there exists a unique solution of \eqref{LapFlow} defined for a short time $t\in[0,\varepsilon)$, 
with $\varepsilon$ depending on the initial datum $\f$. 
The evolution equation of $g_{\f(t)}$ can be deduced from that of $\f(t)$ (cf.~\cite{Bry,Kar2}). It is
\begin{equation}\label{metricevtd}
\ddt g_{\f(t)} = -2\,\mbox{Ric}(g_{\f(t)}) -\frac13\,|\tau|_{\f(t)}^2\,g_{\f(t)}-\widetilde{\tau},
\end{equation}
where $|\tau|_{\f}^2 =g_{\f}(\tau,\tau)= \frac12\tau_{ij}g_\f^{ik}g_\f^{jl}\tau_{kl}$, and $\widetilde\tau$ is the symmetric $(2,0)$-tensor defined locally by 
$\widetilde{\tau}_{ij}\coloneqq \tau_{ik}g_\f^{kr}\tau_{rj}$, $g_\f^{kr}$ being the components of the inverse matrix of $g_{\f}$. 
Hence, \eqref{metricevtd} corresponds to the Ricci flow for $g_{\f(t)}$ up to lower order terms. 

Explicit examples of solutions of the Laplacian flow were worked out in \cite{FFM,Lau1,Lau2,Lau3,Nic}.
All of them are obtained on homogeneous spaces, and most of them consist of {\em self-similar solutions} of \eqref{LapFlow}, namely solutions of the form 
\[
\f(t) = \varrho(t)F_t^*\f,
\]
where $F_t\in{\rm Dif{}f}(M^7)$ is a one-parameter family of diffeomorphisms, and $\varrho(t)=\left(1+\frac23\lambda t\right)^{\frac32}$, $\lambda\in\R$, is a scaling factor. 
It can be shown that the solution of the Laplacian flow is self-similar if and only if the initial datum $\f$ satisfies the equation
\begin{equation}\label{LapSol}
\Delta_\f\f = \mathcal{L}_X\f + \lambda\f,
\end{equation}
for some vector field $X$ on $M^7$ (see e.g.~\cite{Lin,LotWei1}). A closed $\G_2$-structure for which \eqref{LapSol} holds is called a {\em Laplacian soliton}. 
Depending on the sign of $\lambda$, a Laplacian soliton is said to be {\em shrinking} ($\lambda<0$), {\em steady} ($\lambda=0$), or {\em expanding} ($\lambda>0$), 
and the corresponding self-similar solution exists on the maximal time interval  
$\left(-\infty,-\frac{3}{2\lambda}\right)$, $(-\infty,+\infty)$, $\left(-\frac{3}{2\lambda},+\infty\right)$, respectively. 
We shall recall some facts about Laplacian solitons on Lie groups in Section \ref{ExSect}, while we refer the reader to \cite{Lau2,Lin,LotWei1} 
for more details on this topic.

\subsection{The Laplacian flow starting from a closed warped G$_{\mathbf 2}$-structure}
From now on, we assume that $M^6$ is a compact, connected 6-manifold endowed with an $\SU(3)$-structure $(\omega,\psip)$. 
Consider a positive function $f\in\mathcal{C}^\infty(M^6)$ and the warped $\G_2$-structure $\f=f\,\omega\W ds+\psip$ on $M^7=M^6\times\Su$. 
When $\f$ is closed, it is natural to ask under which conditions the solution $\f(t)$ of the Laplacian flow \eqref{LapFlow} starting from it remains a closed warped $\G_2$-structure. 
If this happens, then we have
\begin{equation}\label{phitWP}
\f(t) = f(t)\,\omega(t)\W ds+\psip(t),
\end{equation}
for a family of $\SU(3)$-structures $(\omega(t),\psip(t))$ on $M^6$ and a $t$-dependent function $f(t)\in\mathcal{C}^\infty(M^6)$ such that 
$\omega(0)=\omega$, $\psip(0)=\psip$, and $f(0)=f$. 

\begin{notation}
Since the $\SU(3)$-structure $(\omega(t),\psip(t))$ depends on $t$, so does the decomposition of the spaces $\Omega^k(M^6)$. 
We shall emphasize this fact by using the notation $\Omega^k_r(M^6,t)$ for the summands of such decomposition, 
while we reserve the symbol $\Omega^k_r(M^6)$ for $\Omega^k_r(M^6,0)$. 
The tensors defined by the pair $(\omega(t),\psip(t))$ as explained in Section \ref{settingsect} will be denoted by $J_t$, $h(t)$, $\psim(t)$, $dV_{h(t)}$. 
Moreover, we will denote by $d_6^{*_t}$ the formal adjoint of $d_6$ on $\left(M^6,h(t),dV_{h(t)}\right)$. 
\end{notation}

As the warped $\G_2$-structure \eqref{phitWP} is closed, it follows from \eqref{structueqnsu3} that the differential forms $\omega(t)$, $\psip(t)$, $\psim(t)$ satisfy 
\[
\begin{cases}
d_6\omega(t) = \theta(t)\W\omega(t),\\
d_6\psip(t) = 0,\\
d_6\psim(t) = \Wd(t)\W\omega(t),
\end{cases}
\]
where $\theta (t)=-\frac12\,J_t\,d_6^{*_t} \omega (t)=-d_6\log(f (t))$, and $\Wd(t)\in\Omega^2_8(M^6,t)$ for each $t$. 

Let us assume that the solution of \eqref{LapFlow} starting from a warped closed $\G_2$-structure stays warped under the Laplacian flow. 
Following the approach used in \cite{LotSes, Tran} to study the Ricci flow on warped products, 
we can exploit the results of Section \ref{warpedclosedg2sect} to deduce the evolution equations of $\omega(t)$, $\psip(t)$, $h(t)$ and $f(t)$ from the flow of $\f(t)$. 
Let us begin with those of $\omega(t)$ and $\psip(t)$. 

Differentiating both sides of \eqref{phitWP} with respect to $t$, we get
\[
\ddt\f(t) = \left(\ddt\,f(t)\,\omega(t)+f(t)\ddt\omega(t)\right)\W ds+\ddt\psip(t).
\]
Comparing this with the expression \eqref{HLWP} of $\Delta_\f\f$ gives
\begin{equation}\label{omgevo}
\ddt\omega (t) = d_6d_6^{*_t} \omega(t)+\frac12\,\left(J_t d_6^{*_t} \omega (t)\right)\W d_6^{*_t} \omega (t)-\left(\frac{1}{f (t)}\ddt\,f (t)\right)\omega (t),
\end{equation}
and
\begin{equation}\label{psipevo}
\ddt\psip  (t)= \Delta_{h(t)} \psip (t) +d_6\left(\iota_{\theta (t)^{\sharp_{h(t)}}}\psip (t)\right).
\end{equation}

Starting from $\ddt\left(\omega(t)\W\psip(t)\right)=0$, by \eqref{omgevo}, \eqref{psipevo}, \eqref{staralphapsim}, point \ref{idi}) of Lemma \ref{su3identities}, 
and using the identity $d_6^{*_t}\omega(t)=2\,J_t\theta(t)$, we obtain
\begin{eqnarray*}
0 	&=& \ddt\omega(t)\W\psip(t) + \omega(t)\W\ddt\psip(t) \\ 
	&=&	d_6d_6^{*_t} \omega(t)\W\psip(t) + \omega(t)\W\Delta_{h(t)}\psip(t) + \omega(t)\W d_6*_{h(t)}(\theta(t)\W\psim(t))\\
	&=& 2\,d_6(J_t\theta(t)\W\psip(t)) - d_6\omega(t)\W\Wd(t)+d_6(\omega(t)\W *_{h(t)}(\theta(t)\W\psim(t)))\\
	&=& 2\,d_6(\theta(t)\W\psim(t)) - \theta(t)\W d_6\psim(t) + d_6(\theta(t)\W\psim(t))\\
	&=& 4\,d_6\left(\theta(t)\W\psim(t)\right).
\end{eqnarray*}
Consequently, for any $t$ it must hold 
\begin{equation}\label{compflow}
0 = \theta(t)\W d_6\psim(t) = \theta(t)\W\Wd(t)\W\omega(t) = d_6\omega(t)\W\Wd(t). 
\end{equation}

Let us focus now on the evolution equation \eqref{metricevtd} for the metric $g_{\f(t)}=h(t)+[f(t)]^2 ds^2$ induced by $\f(t)$. 
First of all, observe that 
\[
\ddt g_{\f(t)} = \ddt h(t) +2f(t)\ddt f(t)\,ds^2. 
\]
In order to obtain the evolution equations for $h(t)$ and $f(t)$ following the strategy used before, 
we need to  write the right-hand side of \eqref{metricevtd} as the sum of 
a symmetric $(2,0)$-tensor on $M^6$ and a summand of the form $u\,ds^2$, for some function $u\in\mathcal{C}^\infty(M^6)$.  
By standard results on warped product metrics (see e.g.~\cite{On}), the Ricci tensor of $g_\f = h +f^2ds^2$ is
\[
\Ric(g_\f) = \left(\Ric(h) -\frac{1}{f}\,\Hess(f)\right)  + (-f\Delta f )ds^2,
\]
where $\Delta = -\Delta_h$ is the Laplacian acting on functions, and $\Hess$ is the Hessian on $(M^6,h)$. 
Hence
\[
-2\,\Ric(g_\f)-\frac13|\tau|_\f^2\,g_\f =  \left(-2\,\Ric(h)+\frac{2}{f}\,\Hess(f) -\frac13|\tau|_\f^2\,h \right) + \left(2f\Delta f -\frac13|\tau|_\f^2f^2\right)ds^2.
\]
What we need to check is then whether
\[
\widetilde{\tau} = \widetilde{\tau}_6+\widetilde{\tau}_{77}\,ds^2, 
\]
where the first summand is given locally by $\widetilde{\tau}_6 \coloneqq \sum_{i,j=1}^6 \widetilde{\tau}_{ij}dx^idx^j \in \mathcal{S}^2(M^6)$, 
$(x^1,\ldots, x^6)$ being local coordinates on $M^6,$ and $x^7=s$.  
This happens if and only if the following components of $\widetilde{\tau}$ vanish for each $i=1,\ldots,6$:
\begin{equation}\label{wlmetric}
\widetilde{\tau}_{i7} = \widetilde{\tau}_{7i} = \sum_{k,j=1}^6 \tau_{ik}h^{kj}\tau_{j7},
\end{equation}
where $h^{kj}$ are the components of the inverse of $h$ in local coordinates. 
Comparing the above expression with \eqref{tau2WP}, and using point \ref{idiii}) of Lemma \ref{su3identities}, we obtain that \eqref{wlmetric} 
are proportional to the components of the following 1-form
\begin{equation}\label{oneformcontr}
\iota_{\left[f*_h(\theta\W\omega^2)\right]^{\sh}}\left(\Wd+*_h(\theta\W\psim)\right).
\end{equation}
Applying identity \eqref{wedgestarcontraction} to \eqref{oneformcontr}, we get 
\[
*_h\left[f*_h(\theta\W\omega^2)\W\left(*_h\Wd + \theta\W\psim \right) \right].
\]
Using points \ref{idi}) and \ref{idiii}) of Lemma \ref{su3identities}, the identity $*_h\Wd=-\Wd\W\omega$, and recalling that both $\omega$ and $\Wd$ are of type $(1,1)$ with 
respect to $J$, we can rewrite the above 1-form as
\[
2\, f *_hJ\left(\theta\W\Wd\W\omega\right).
\]
Therefore, when we consider the metric induced by the solution \eqref{phitWP} of the Laplacian flow, the components \eqref{wlmetric} of $\widetilde{\tau}$ 
must be zero by \eqref{compflow}.
Summing up,  we have
\[
\ddt h(t) = -2\,\Ric(h(t))+\frac{2}{f(t)}\,\Hess(f(t)) -\frac13\,|\tau|_{\f(t)}^2\,h(t) -\widetilde{\tau}_6(t),
\]
and
\[
\ddt f(t) = \Delta f(t) -\frac16|\tau|_{\f(t)}^2f(t)  -\frac{1}{2f(t)}\widetilde{\tau}_{77}(t). 
\]
Observe that the evolution equations for $h(t)$ and $f(t)$ are both defined on $M^6,$ since by \eqref{normtau2}
\[
|\tau|_{\f(t)}^2 	= \left|\Wd(t)\right|_{h(t)}^2+6\left|\theta(t)\right|_{h(t)}^2,
\]
and by \eqref{twoeqn}
\[
\widetilde{\tau}_{77}(t) = \sum_{i,j=1}^6\tau_{7i}h^{ij}\tau_{j7} = -f(t)^2\left|\theta(t)\W\omega(t)^2\right|_{h(t)}^2 = -4\,f(t)^2\left|\theta(t)\right|_{h(t)}^2, 
\]
We can summarize the previous results in the following. 
\begin{proposition}\label{summinducedflow}
Assume that $\f(t)=f(t)\,\omega(t)\W ds+\psip(t)$ is the solution of the Laplacian flow on $M^6\times\Su$ starting from the closed warped $\G_2$-structure 
$\f=f\,\omega\W ds+\psip$ at $t=0$, where $(\omega,\psip)$ is an $\SU(3)$-structure on $M^6$ satisfying \eqref{structueqnsu3}. 
Then, the evolution of the family of $\SU(3)$-structures $(\omega(t),\psip(t))$ on $M^6$ is described by the equations
\begin{equation}\label{flowsu3t}
\left\{ 
\renewcommand\arraystretch{1.4}
\begin{array}{ccl}
\ddt\omega (t) &=& d_6d_6^{*_t} \omega(t)+\frac12\,(J_t d_6^{*_t} \omega (t))\W d_6^{*_t} \omega (t)-\left(\frac{1}{f (t)}\ddt\,f (t)\right)\omega (t),\\
\ddt\psip  (t) &=& \Delta_{h(t)} \psip (t) +d_6\left(\iota_{\theta (t)^{\sharp_{h(t)}}}\psip (t)\right),
\end{array}
\right.
\end{equation}
where $\theta (t)=-d_6\log(f (t))$, and
\begin{equation}\label{flowf}
\ddt f(t) = \Delta f(t)  +  \left(\left|\theta(t)\right|_{h(t)}^2-\frac16|\Wd(t)|_{h(t)}^2\right) f(t). 
\end{equation}
Moreover, the 1-form $\theta(t)$ satisfies the equation
\[
\theta(t)\W d_6\psim(t)=0.
\]
Finally, the metric $h(t)$ induced by $(\omega(t),\psip(t))$ evolves as follows
\[
\ddt h(t) = -2\,\Ric(h(t))+\frac{2}{f(t)}\,\Hess(f(t)) -\frac13\,|\tau|_{\f(t)}^2\,h(t) -\widetilde{\tau}_6(t).
\]
\end{proposition}

By reversing the argument, it is possible to check that the existence of a family of $\SU(3)$-structures $(\omega(t),\psip(t))$ satisfying \eqref{flowsu3t} 
and a $t$-dependent function $f(t)$ solving \eqref{flowf} implies that the family of closed warped $\G_2$-structures $\f(t) = f(t)\,\omega(t)\W ds+\psip(t)$ 
is the solution of the Laplacian flow on $M^6\times\Su$ starting from $\f(0) = f\,\omega\W ds+\psip$. 
However, it seems difficult to establish whether solutions of an initial value problem for the evolution equations \eqref{flowsu3t} and \eqref{flowf} always exist.

\section{Laplacian flow on Riemannian product manifolds}\label{lapflowproduct}
In this section, we assume that the warping function $f$ in \eqref{g2strwp} is identically constant,  $f\equiv a\in\R_\+$. 
Under this hypothesis, $(M^7=M^6\times\Su,h+a^2ds^2)$ is a Riemannian product manifold, and the 3-form $\f=a\,\omega\W ds+\psip$ defines a closed $\G_2$-structure on it 
if and only if the $\SU(3)$-structure $(\omega,\psip)$ is symplectic half-flat. 
Recall that, in such a case, 
\[
d_6\omega=0,\quad d_6\psip=0,\quad d_6\psim=\Wd\W\omega,
\]
where $\Wd\in\Omega^2_8(M^6)$. 
The intrinsic torsion form $\tau\in\Omega^2_{14}(M^7)$ of the closed $\G_2$-structure coincides with $\Wd$, 
and its norm is proportional to the scalar curvature of the metric $h$ induced by $(\omega,\psip)$. 
Indeed, $|\tau|_{\f}^2 =\left|\Wd\right|^2_h$ (cf.~\eqref{normtau2}), and 
\begin{equation}\label{ScalSHF}
\mbox{Scal}(h) = -\frac12\left|\Wd\right|^2_h,
\end{equation}
by \cite[Thm.~3.4]{BedVez}. 
In the next lemma, we collect some useful properties of $\Wd$. 
\begin{lemma}\label{dwdexpr}
Let $(\omega,\psip)$ be a symplectic half-flat $\SU(3)$-structure. Then, the intrinsic torsion form $\Wd$ is coclosed, and its exterior derivative has the following expression
\begin{equation}\label{dWdexpr}
d_6\Wd = \frac{\left|\Wd\right|_h^2}{4}\psip +\gamma,
\end{equation}
for some $\gamma\in\Omega^3_{12}(M^6)$.  
As a consequence, the following identities hold
\begin{equation}\label{iddw2}
\begin{array}{rcl}
d_6\Wd \wedge \psip  &=& *_h d_6\Wd \wedge \psim  = 0, \\
d_6 \Wd \wedge \psim  &=&  \psip \W*_h d_6\Wd =  \left| \Wd \right|_h^2 *_h1.
\end{array}
\end{equation}
Moreover, if $\gamma$ vanishes identically, then $|\Wd|_h$ is constant.  
\end{lemma}
\begin{proof}
In point \ref{rempropWd}) of Remark \ref{W2W4remark}, we observed that $\Wd=d_6^*\psip$. Hence, the 2-form $\Wd$ is clearly coclosed. 
By the decomposition \eqref{3formdec6} of $\Omega^3(M^6)$, there exist unique functions $k^+,k^-\in\mathcal{C}^\infty(M^6)$ and differential forms $\gamma\in\Omega^3_{12}(M^6)$, 
$\beta\in\Omega^1(M^6)$ such that 
\[
d_6\Wd = k^+\psip+k^-\psim+\gamma+\beta\W\omega.
\]
Since $d_6\Wd\W\omega = d_6(d_6\psim)=0$, we have $\beta\W\omega^2=0$, which implies $\beta=0$. 
Differentiating the identity $\Wd\W\psip=0$, we get $k^-=0$. 
Finally, taking the exterior derivative of $\Wd\W\psim=0$, we obtain
\[
0 = d_6\Wd\W\psim + \Wd\W\Wd\W\omega = k^+\psip\W\psim -\Wd\W*_h\Wd = \left(\frac23\,k^+ -\frac16 \left|\Wd \right|^2_h\right)\omega^3,
\]
from which the second assertion follows. 
When $\gamma\equiv0$, we have 
\[
0 = d_6(d_6\Wd)= \frac{1}{4}d_6\left|\Wd\right|_h^2\W\psip.
\]
Since wedging 1-forms by $\psip$ is injective, this implies that $\left|\Wd\right|_h$ is constant. 
\end{proof}

\begin{remark}
Note that the identity \eqref{dWdexpr} also holds for the warped solution $\f(t)=f(t)\,\omega(t)\W ds+\psip(t)$ of the Laplacian flow even when $f(t)$ is not constant. 
This is an easy consequence of \eqref{compflow}. 
\end{remark}

Using the results of the previous sections, we immediately obtain the following. 
\begin{proposition}\label{propflowprod}
Assume that $\f(t)=f(t)\,\omega(t)\W ds+\psip(t)$ is the solution of the Laplacian flow on $M^6\times\Su$ starting from the closed $\G_2$-structure $\f=a\,\omega\W ds+\psip$ at $t=0$. 
Then, the function $f(t)$ is constant on $M^6$ for each $t$, i.e., $d_6f(t)=0$, if and only if the $\SU(3)$-structure $(\omega(t), \psip(t))$ is symplectic half-flat. 
In this case, the evolution equations are
\begin{equation}\label{evsysfcnst}
\renewcommand\arraystretch{1.4}
\left\{
\begin{array}{l}
\ddt\omega(t) 	= -\left(\frac{1}{f (t)}\frac{\mathrm d}{\mathrm d t}f (t)\right)\omega (t),\\
\ddt\psip(t) 	= \Delta_{h(t)}\psip(t),
\end{array}
\right.
\end{equation}
and
\begin{equation}\label{evolutionf}
\frac{\mathrm d}{\mathrm d t} f(t) =  -\frac16\, \left|d_6^{*_t}\psip(t)\right|_{h(t)}^2\,f(t). 
\end{equation}
In particular, \eqref{evolutionf} implies that $\left|d_6^{*_t}\psip(t)\right|^2_{h(t)}= -2\,{\rm Scal}(h(t))$ is constant for each $t$. 
\end{proposition}
\begin{proof}
The assertion follows from point \ref{fconstiffshf}) of Remark \ref{W2W4remark}, and from Proposition \ref{summinducedflow}.
\end{proof}

Observe that the solution of the first equation in \eqref{evsysfcnst} is 
\begin{equation}\label{omgt}
\omega(t) = \exp\left(-\int_0^t\frac{\mathrm d}{\mathrm d r}\log(f(r))dr\right)\omega(0) = \frac{f(0)}{f(t)}\,\omega(0). 
\end{equation}
Consequently, the 2-form $\omega(t)$ is stable and closed for each $t$, and it is completely determined once that $f(t)$ is known. 

Let us now focus on the evolution equations for the function $f(t)$ and the 3-form $\psip(t)$. Changing our point of view, 
it would be interesting to study the short-time existence and uniqueness of the solution for the problem
\begin{equation}\label{systcompletefomgpsip}
\begin{cases}
\ddt\psip(t) 	= \Delta_{h(t)}\psip(t),\\
\omega(t)\W\psip(t)=0,\\
\frac{\mathrm d}{\mathrm d t} f(t) =  \frac{1}{3}\,{\rm Scal}(h(t))f(t),\\
\psip(0)=\psip,\\
f(0) = a,
\end{cases}
\end{equation}
where $(\omega,\psip)$ is a symplectic half-flat $\SU(3)$-structure whose associated metric has constant scalar curvature, 
$\omega(t)$ is given by \eqref{omgt} with $\omega(0)=\omega$, and $a\in\R_\+$. 
Since $\psip$ is a stable closed 3-form, it would be desirable that these properties were satisfied also by a possible solution $\psip(t)$. 
If this happens, then we would get a family of symplectic half-flat $\SU(3)$-structures $(\omega(t),\psip(t))$ on $M^6$ together with a function $f(t)$ such that 
the family of closed $\G_2$-structures $\f(t) = f(t)\,\omega(t)\W ds+\psip(t)$ on $M^6\times\Su$ 
is the solution of the Laplacian flow starting from the closed $\G_2$-structure $\f(0)=a\,\omega\W ds+\psip$. 
Following the general approach outlined in \cite[Sect.~3.1]{BedVez2}, it seems natural to focus on the problem 
\begin{equation}\label{psipflow}
\renewcommand\arraystretch{1.4}
\left\{
\begin{array}{l}
\ddt\psip(t) 	= \Delta_{h(t)}\psip(t),\\
\omega(t)\W\psip(t)=0,\\
d_6\psip(t)=0,\\
\psip(0)=\psip,
\end{array}
\right.
\end{equation}
and look for solutions in the space
\[
\mathcal{O}\coloneqq \left\{\psip+d_6\beta\st\beta\in\Omega^2(M^6),~d_6\beta\W\omega=0 \right\}\cap\Omega_\+^3(M^6).
\]
In this way, the definition of the Hodge Laplacian with respect to the metric $h(t)$ induced by $\omega(t) = \frac{a}{f(t)}\,\omega$ and $\psip(t)\in\mathcal{O}$ makes sense, 
and one can investigate whether it is possible to prove short-time existence and uniqueness of the solution of \eqref{psipflow} using \cite[Thm.~3.2]{BedVez2}. 

In the remainder of this section, we show that there exist a solution of \eqref{psipflow} belonging to $\mathcal{O}$ when the initial datum $(\omega,\psip)$ 
is a symplectic half-flat $\SU(3)$-structure satisfying some additional properties, while we reserve the study of the general problem for future work. 
Let us begin considering an explicit example. 
\begin{example}\label{firstexsolv}
Let $\frg\coloneqq A_{5,7}^{-1,-1,1}\oplus\R$ be the six-dimensional, decomposable, solvable Lie algebra with structure equations 
\[
(e^{15},-e^{25},-e^{35},e^{45},0,0).
\]
Recall that the above notation means that there exists a basis $\{e_1,\ldots,e_6\}$ of $\frg$ whose dual basis $\{e^1,\ldots,e^6\}$ satisfies
\[
de^1=e^{15},~de^2=-e^{25},~de^3=-e^{35},~de^4=e^{45},~de^5=0,~de^6=0, 
\]
where $d$ is the Chevalley-Eilenberg differential of $\frg$, and $e^{ij\cdots}$ is a shorthand for the wedge product $e^i\W e^j\W \cdots$.
It is known (see \cite{FMOU,FrSchD}) that on $\frg$ there exists a symplectic half-flat $\SU(3)$-structure defined by the differential forms 
\[
\omega = -e^{13}+e^{24}+e^{56},\quad \psip = -e^{126}-e^{145}-e^{235}-e^{346}.
\]
The almost complex structure $J$ induced by $\psip$ and the volume form $\frac{\omega^3}{6}$ is the following
\[
Je_1=-e_3,\quad Je_2=e_4,\quad Je_5=e_6,
\]
the inner product associated with $(\omega,\psip)$ is $h = \sum_{i=1}^6(e^i)^2$, and the imaginary part of the complex volume form $\Psi=\psip+i\psim$ is
\[
\psim = e^{125}-e^{146}-e^{236}+e^{345}.
\]
Consequently, $\Wd = 2\,e^{14}-2\,e^{23}.$
Notice that $\Delta_h\Wd = 4\,\Wd$, and that $\left|d_6\Wd\right|_h^2 = 4\left|\Wd\right|_h^2$. 

Let $f(t)$ be a positive function such that $f(0)=a\in\R_{\scriptscriptstyle+}$, and let $\omega(t)=\frac{a}{f(t)}\,\omega$. 
We consider the following Ansatz for the solution of \eqref{psipflow}:
\[
\psip(t) = \psip + k(t)\,d\Wd,
\]
where $k(t)$ is an unknown function satisfying $k(0)=0$. Clearly, $\psip(t)$ is closed, and $\omega(t)\W\psip(t)=0$.  
Moreover, $\psip(t)$ is stable for each value of $k(t)$ but $-\frac14$, since
\[
P(\psip(t)) = -4(1+4k(t))^2. 
\]
Therefore, the pair $(\omega(t),\psip(t))$ defines a family of symplectic half-flat $\SU(3)$-structures 
provided that the symmetric bilinear form $h(t)\coloneqq \omega(t)(\cdot,J_t\cdot)$ is positive definite, the normalization condition is satisfied, and $k(t)\neq-\frac14$. 
The matrix associated with $h(t)$ with respect to the basis  $\{e_1,\ldots,e_6\}$ is 
\[
\frac{1+4k(t)}{|1+4k(t)|}\frac{a}{f(t)}\,\mbox{diag}\left(1,1,1,1,1+4k(t),\frac{1}{1+4k(t)}\right). 
\]
Since both $a$ and $f(t)$ are positive, it follows that $h(t)$ is positive definite if and only if $1+4k(t)>0$. 
We can determine explicitly the expression of $\psim(t)$, obtaining
\[
\psim(t) = (1+4k(t))e^{125}-e^{146}-e^{236}+(1+4k(t))e^{345}.
\] 
Imposing the normalization condition \eqref{normcond}, we get
\[
k(t) = -\frac14\frac{f(t)^3-a^3}{f(t)^3}.
\]
Now, we can compute the expression of $\Wd(t)$ and its norm:
\[
\Wd(t) = 2\,\frac{f(t)}{a}e^{14} -2\,\frac{f(t)}{a}e^{23},\qquad \left|\Wd(t)\right|^2_{h(t)} = 8\,\frac{f(t)^4}{a^4}.
\]
From the equation $\ddt\psip(t) = \Delta_{h(t)}\psip(t) = d\Wd(t)$, which is now an ODE, and the identity relating $k(t)$ and $f(t)$, 
we deduce that $f(t)$ has to solve the Cauchy problem
\[
\begin{cases}
\frac{\mathrm d}{{\mathrm d}t} f(t) = -\frac{4}{3}\frac{f(t)^5}{a^4},\\
f(0)=a.
\end{cases}
\]
A standard computation gives
\[
f(t) = a\left(\frac{16}{3}t+1\right)^{-\frac14}.
\]
Observe that the same ODE for $f(t)$ is obtained considering the equation \eqref{evolutionf}.

Summing up, the family of symplectic half-flat $\SU(3)$-structures $(\omega(t),\psip(t))$ starting from $(\omega,\psip)$ at $t=0$ and solving the evolution equations \eqref{evsysfcnst} 
is given by
\begin{eqnarray*}
\omega(t) 	&=& \left(\frac{16}{3}t+1\right)^{\frac14}(-e^{13}+e^{24}+e^{56}),\\
\psip(t)	&=& -e^{126}-\left(\frac{16}{3}t+1\right)^{\frac34}e^{145}-\left(\frac{16}{3}t+1\right)^{\frac34}e^{235}-e^{346},\qquad t\in\left(-\frac{3}{16},+\infty\right).
\end{eqnarray*}
Moreover, the product algebra $\frg\oplus\R e_7$ admits a family of closed $\G_2$-structures 
\[
\f(t) \coloneqq f(t)\,\omega(t)\W e^7 +\psip(t) = a\, \omega\W e^7 + \psip(t), 
\]
which is the solution of the Laplacian flow starting from $\f(0) = a\,\omega\W e^7+\psip$. 
\end{example}

In the light of the previous example, we may investigate whether there exist some conditions on the initial datum $(\omega,\psip)$ guaranteeing that the Ansatz 
$\psip(t)=\psip+k(t)d_6\Wd$ for the solution of \eqref{psipflow} is valid in greater generality. 
The next result gives a positive answer. Moreover, we shall see in Section \ref{ExSect} that there are many examples where it applies. 

\begin{theorem}\label{thmlapshf}
Let $M^6$ be a compact, six-dimensional manifold endowed with a symplectic half-flat $\SU(3)$-structure $(\omega,\psip)$. 
Assume that both $\left|\Wd\right|_h$ and $\left|d_6\Wd\right|_h$ are constant, that $d_6\Wd\W\Wd=0$, 
and that there exists a positive real number $c$ such that
\[
\Delta_h\Wd = c\,\Wd. 
\]
Then, $c \geq\frac{\left| \Wd \right|_h^2}{4}$, and for any real constant $a>0$, the solution of \eqref{systcompletefomgpsip} satisfying $f(0)=a$ and 
$(\omega(0),\psip(0))=(\omega,\psip)$ is given by the function
\[
f(t) = a\left(\frac{6c-|\Wd|_h^2}{3}\, t+1 \right)^{\frac{|\Wd|_h^2}{2|\Wd|_h^2-12c}},
\]
and the family of symplectic half-flat $\SU(3)$-structures on $M^6$
\[
\omega(t) = \frac{a}{f(t)}\,\omega,\qquad \psip(t) = \psip  + k(t)\,d_6\Wd,
\]
where
\[
k(t) = \frac{1}{c}\left[\left(\frac{a}{f(t)}\right)^{\frac{6c}{|\Wd|_h^2}}-1 \right].
\]
Consequently, the solution of the Laplacian flow
\[
\begin{cases}
\ddt \f(t) = \Delta_{\f(t)}\f(t),\\
d_7 \f(t)=0,\\
\f(0)=a \,\omega \W ds+\psip,
\end{cases}
\]
on the compact 7-manifold $M^6\times\Su$ is given by
\[
\f(t)=f(t)\,\omega(t)\W ds+\psip(t),
\]
and its maximal interval of existence is $\left(T,+\infty\right)$, where $T\coloneqq \frac{3}{|\Wd|_h^2-6\,c}<0$. 
\end{theorem}
\begin{proof}  
First of all, let us observe that the hypothesis on $\Wd$ imply 
\begin{equation}\label{propdw2w2}
\left| d_6 \Wd \right|_h^2 = c{\left| \Wd \right|_h^2}.
\end{equation}
Indeed, denoted by $(\cdot,\cdot)$ the $L^2$ inner product on $\Omega^k(M^6)$ and by $\left\|\cdot\right\|$ the associated norm, we have
\[
\left\|d_6\Wd \right\|^2 = (d_6\Wd,d_6\Wd) = (d_6^*d_6\Wd,\Wd) = (\Delta_h\Wd,\Wd) = c \left| \Wd \right|_h^2 \mathrm{Vol}(M^6),
\]
and
\[
\left\|d_6\Wd \right\|^2 = \left| d_6\Wd \right|_h^2 \mathrm{Vol}(M^6), 
\]
since $\left| d_6\Wd \right|_h$ is constant. This proves \eqref{propdw2w2}. Furthermore, from Lemma \ref{dwdexpr}, we get
\[
\left| d_6 \Wd \right|_h^2 = \frac{\left| \Wd \right|_h^4}{4} + |\gamma|_h^2 \geq \frac{\left| \Wd \right|_h^4}{4},
\]
from which follows that $c = \frac{\left| d_6 \Wd \right|_h^2}{\left| \Wd \right|_h^2} \geq\frac{\left| \Wd \right|_h^2}{4}$. 

We already know that the 2-form $\omega(t) =   a f(t)^{-1} \omega$ is stable and closed for every $t$. 
Let us consider the 3-form
\[
\psip (t) \coloneqq \psip  + k(t)  \,  d_6 \Wd,
\]
where $k(t)$ is a function depending only on $t$ and such that $k(0)=0$. 
As $\psip(0)=\psip\in\Omega^3_+(M^6)$, we can assume that $\psip(t)\in\Omega^3_+(M^6)$ for $t\in[0,\varepsilon)$, with $\varepsilon>0$ sufficiently small.  
Since $k(t)$ is constant on $M^6$, the 3-form $\psip (t)$ is closed. 
Moreover, the compatibility condition $\omega(t)\W\psip (t) = 0$ is satisfied by \eqref{dWdexpr}. 
Hence, the pair $(\omega(t),\psip(t))$ defines a family of symplectic half-flat $\SU(3)$-structures on $M^6.$  
We claim that 
\[
\psim(t) = \frac{a^3}{f(t)^3} \left( \psim - \frac{k(t)}{1 + c\,k(t)}    *_h d_6 \Wd\right). 
\]
Indeed, since $\omega(t)$ is proportional to $\omega$, the identity \eqref{dWdexpr} implies that $\psim(t)\W\omega(t)=0$. 
Moreover, using identities \eqref{iddw2} and \eqref{propdw2w2}, we have 
\[
\psip (t)\wedge \psim(t) = \frac{2}{3}\, \frac{a^3}{f(t)^3}\, \omega^3 = \frac{2}{3}\, \omega(t)^3.
\]
Our claim follows then from the results of \cite[Sect.~2.1]{BedVez}. Now, there exists a unique $\Wd(t)\in\Omega^2_8(M^6,t)$ such that
\[
d_6 \psim (t) =  \Wd (t) \wedge \omega (t).
\]
By the decomposition \eqref{4formsdec6} of $\Omega^4(M^6)$, we can write
\[
d_6(*_h d_6 \Wd) = \beta_0\,  \omega^2 + \beta_2 \wedge \omega + \beta_1 \wedge \psip,
\]
for unique $\beta_0 \in {\mathcal C}^{\infty} (M^6)$, $\beta_2 \in \Omega^2_8(M^6)$ and $\beta_1 \in \Omega^1 (M^6).$ 
The hypothesis $\Delta_h\Wd = c\,\Wd$ implies that both $\beta_0$ and $\beta_1$ are zero, and that $\beta_2=c\,\Wd$. 
Consequently, we get
\[
d_6 \psim (t) = \frac{a^2}{f(t)^2}\left(\frac{1}{1+c\,k(t)}\right)\Wd \W \omega(t).
\]
Using $\psip(t) = \psip  + k(t)\,d_6\Wd$ and $\omega(t)= a f(t)^{-1} \omega$, we see that the equation $d\Wd\W\Wd=0$ implies $\Wd\W\psip(t)=0$ and that 
$\Wd\W\omega(t)^2=0$, whence
\[
\Wd(t)\coloneqq\frac{a^2}{f(t)^2}\left(\frac{1}{1+c\,k(t)}\right)\Wd \in\Omega^2_8(M^6,t).
\] 

Now, from the equation $\ddt\psip(t) = \Delta_{h(t)}  \psip (t) = d_6 \Wd(t)$, we obtain  
\[
\dt k(t) =   \frac{a^2}{f(t)^2}\left(\frac{1}{1+c\,k(t)}\right).
\]
Albeit we have not computed the explicit expression of the metric $h(t)$, we can deduce the norm of $\Wd(t)$ from the identity 
\[
-\Wd(t)\W\Wd(t)\W\omega(t) = \left|\Wd(t)\right|^2_{h(t)}\frac{\omega(t)^3}{6}.
\]
It is
\[
\left|\Wd(t)\right|^2_{h(t)} = \frac{a^2}{f(t)^2}\left(\frac{1}{1+c\,k(t)}\right)^2\left|\Wd\right|^2_{h}.
\]
Hence, the evolution equation \eqref{evolutionf} of the function $f(t)$ becomes 
\[
\dt f(t) = -\frac16 \frac{a^2}{f(t)}\left(\frac{1}{1+c\,k(t)}\right)^2\left|\Wd\right|^2_{h}. 
\]
Thus, we have to solve the following initial value problem:
\[
\renewcommand\arraystretch{1.4}
\left\{
\begin{array}{l}
\dt f(t) = -\frac16 \frac{a^2}{f(t)}\left(\frac{1}{1+c\,k(t)}\right)^2\left|\Wd\right|^2_{h},\\
\dt k(t) =   \frac{a^2}{f(t)^2}\left(\frac{1}{1+c\,k(t)}\right),\\
f(0)=a,\\
k(0)=0.
\end{array}
\right.
\]
The two ODEs give
\[
\frac{\mathrm{d} f}{\mathrm{d} k} = -\frac16\frac{\left|\Wd\right|_h^2}{1+c\,k}\,f,
\]
from which we get 
\[
k(t) = \frac{1}{c}\left[\left(\frac{a}{f(t)}\right)^{\frac{6c}{|\Wd|_h^2}}-1 \right].
\]
Substituting this expression into the first ODE of the above system, we obtain the Cauchy problem 
\[
\renewcommand\arraystretch{1.4}
\left\{
\begin{array}{l}
\dt f(t) = -\frac{|\Wd|^2_{h}}{6}\, \left(a^{2-\frac{12\,c}{|\Wd|^2_{h}}}\right) \left(f(t)^{\frac{12\,c}{|\Wd|^2_{h}}-1}\right),\\
f(0)=a.
\end{array}
\right.
\]
A standard computation gives the final expression of $f(t)$.
\end{proof}

\begin{corollary}\label{corproductthm}
Theorem \ref{thmlapshf} holds true for every compact, six-dimensional manifold $M^6$  endowed with a symplectic half-flat $\SU(3)$-structure $(\omega,\psip)$ 
whose intrinsic torsion form $\Wd$ satisfies the equation
\[
d_6\Wd = \frac{\left|\Wd\right|^2_h}{4}\,\psip.
\]
In this case, the function $f(t)$ has the following expression 
\[
f(t) = a\left(\frac{\left|\Wd\right|^2_h}{6}\,t+1\right)^{-1},
\]
and the family of symplectic half-flat $\SU(3)$-structures on $M^6$ evolves only by a rescaling of the initial datum:
\[
\omega(t) = \left(\frac{\left|\Wd\right|^2_h}{6}\,t+1\right)\omega,\quad \psip(t) = \left(\frac{\left|\Wd\right|^2_h}{6}\,t+1\right)^{\frac32}\,\psip.
\]
\end{corollary}
\begin{proof}
By Lemma \ref{dwdexpr}, the norm $\left|\Wd\right|_h$ is constant. It is clear that $d_6\Wd$ has constant norm, too. 
Moreover, $\Wd$ is an eigenform of the Laplacian $\Delta_h$ corresponding to the eigenvalue $\frac{1}{4}\left|\Wd\right|^2_h$. 
Finally, $d\Wd\W\Wd=0$. Thus, all of the hypothesis of Theorem \ref{thmlapshf} are satisfied. 
\end{proof}

We conclude this section with some remarks on the hypothesis of Theorem \ref{thmlapshf}.  
By \eqref{ScalSHF}, $|\Wd|_h$ is constant if and only if the scalar curvature of $h$ is constant. 
This request is motivated by the expression of the evolution equation \eqref{evolutionf} for $f(t)$. 
The closedness of $\Wd\W\Wd$ guarantees that the 2-form $\Wd(t)$ appearing in the proof of the theorem belongs to $\Omega^2_8(M^6,t)$. 
As we will see shortly, there are many examples where this condition is satisfied. 
Nevertheless, it does not need to hold in general, as one can check 
considering for instance the symplectic half-flat $\SU(3)$-structure on $\T^6$ described in Example \ref{extorus}. 
Finally, in the next section we will provide various examples of symplectic half-flat $\SU(3)$-structures whose intrinsic torsion form $\Wd$ is an eigenform of the 
Hodge Laplacian. However, we point out that there exist examples where $\Wd$ does not satisfy this property (cf.~Example \ref{exWdnoteigenform}).

\section{Examples}\label{ExSect}
In this section, we apply the results obtained so far to explicit examples. 

\subsection{Lie groups}
The first class of examples we focus on is given by left-invariant $\SU(3)$- and $\G_2$-structures on simply connected Lie groups. 
It is well-known that the solution of the Laplacian flow starting form a left-invariant closed $\G_2$-structure on a simply connected Lie group $\G$ is left-invariant, 
and that the flow equation is equivalent to an ODE on the Lie algebra $\frg$ of $\G$ (see \cite{FFM,Lau1,Lau2}). 
Hence, short-time existence and uniqueness of the solution of the flow in the class of left-invariant $\G_2$-structures is always guaranteed. 

Recall that an $\SU(3)$-structure on a Lie algebra $\frg$ is defined by a pair of stable forms $\omega\in\Lambda^2(\frg^*)$, 
$\psip\in\Lambda^3_\+(\frg^*)$ satisfying the compatibility condition \eqref{compcond}, the normalization condition \eqref{normcond}, 
and inducing a positive definite inner product $h$ on $\frg$. 
When both $\omega$ and $\psip$ are closed, the pair $(\omega,\psip)$ is a symplectic half-flat $\SU(3)$-structure. 

Starting with a six-dimensional Lie algebra $\frg$ endowed with a symplectic half-flat $\SU(3)$-structure $(\omega,\psip)$, 
the {\em product algebra} $\hat\frg\coloneqq \frg\oplus\R$ has a closed $\G_2$-structure given by 
\begin{equation}\label{phiLie}
\f = \omega\W e^7+\psip,
\end{equation}
where $\R={\rm span}(e_7)$, and  $e^7$ denotes the dual of $e_7$.

An $\SU(3)$-structure on a Lie algebra $\frg$ gives rise to a left-invariant $\SU(3)$-structure on the simply connected Lie group $\G$ with Lie algebra $\frg$. 
Such correspondence is one-to-one.  
When  $\G$ admits a lattice $\Gamma$, the left-invariant differential forms on $\G$ pass to the compact quotient $\Gamma \backslash \G$, 
defining an invariant $\SU(3)$-structure  on it. 
Similarly, the 3-form $\f$ given by \eqref{phiLie} corresponds to a left-invariant $\G_2$-structure on the simply connected Lie group $\hat\G$ with Lie algebra $\hat\frg$, 
and from this it is possible to define an invariant $\G_2$-structure on the compact manifold $\left(\Gamma \backslash \G\right)\times\Su$, whenever a lattice $\Gamma\subset\G$ exists. 

It is well-known that if a Lie group admits a lattice, then it is unimodular (cf.~\cite{Mil}). 
Furthermore, a unimodular symplectic Lie algebra is necessarily solvable (see e.g.~\cite{LicMed}). 

Unimodular, solvable Lie algebras endowed with a symplectic half-flat $\SU(3)$-structure were classified in \cite{FMOU}. 
Using the convention explained in Example \ref{firstexsolv} to denote the structure equations of a Lie algebra, the result can be stated as follows. 
\begin{theorem}[\cite{FMOU}]\label{solvableSHF}
Let $\frg$ be a six-dimensional, unimodular, non-Abelian, solvable Lie algebra. Then, $\frg$ is endowed with a symplectic half-flat $\SU(3)$-structure if and only if 
it is isomorphic to one of the following
\begin{eqnarray*}
\fre(1,1)\oplus\fre(1,1)			&=&	(0,-e^{13},-e^{12},0,-e^{46},-e^{45});\\
\frg_{5,1}\oplus\R				&=& (0,0,0,0,e^{12},e^{13});\\
A_{5,7}^{-1,-1,1}\oplus\R			&=& (e^{15},-e^{25},-e^{35},e^{45},0,0);\\
A_{5,17}^{\alpha,-\alpha,1}\oplus\R	&=& (\alpha e^{15}+e^{25},-e^{15}+\alpha e^{25},-\alpha e^{35}+e^{45},-e^{35}-\alpha e^{45},0,0),~\alpha>0;\\
\frg_{6,N3}					&=& (0,0,0,e^{12},e^{13},e^{23});\\
\frg_{6,38}^{0}					&=& (e^{23},-e^{36},e^{26},e^{26}-e^{56},e^{36}+e^{46},0);\\
\frg_{6,54}^{0,-1}				&=& (e^{16} + e^{35}, -e^{26} + e^{45}, e^{36}, -e^{46}, 0, 0);\\
\frg_{6,118}^{0,-1,-1}				&=& (-e^{16} +e^{25},-e^{15} -e^{26}, e^{36} -e^{45}, e^{35} +e^{46}, 0, 0).
\end{eqnarray*}
\end{theorem}

\begin{remark}
Observe that $\frg_{5,1}\oplus\R$ and $\frg_{6,N3}$ are the only nilpotent Lie algebras among those appearing in the previous theorem.   
Moreover, the first four in the list are decomposable, whereas the remaining ones are indecomposable. 
Finally, all of the corresponding simply connected solvable Lie groups admit a lattice (see \cite{FMOU} and the references therein). 
\end{remark}

In Table \ref{tabshf}, an explicit example of symplectic half-flat $\SU(3)$-structure for each Lie algebra of Theorem \ref{solvableSHF} is given. 
Most of them already appeared in \cite{ConTom,FMOU,FrSchD,TomVez}, while the one on $\frg_{6,54}^{0,-1}$ is new. 

Let $\frg$ be one of the unimodular solvable Lie algebras of Theorem \ref{solvableSHF}, and assume that it is endowed with the symplectic half-flat 
$\SU(3)$-structure described in Table \ref{tabshf}. 
Then, Theorem \ref{thmlapshf} can be adapted in the obvious way to this setting, allowing to find the solution of the 
Laplacian flow starting from the closed $\G_2$-structure given by \eqref{phiLie} on $\hat\frg=\frg\oplus\R$. 

\begin{proposition}
All of the symplectic half-flat $\SU(3)$-structures appearing in Table \ref{tabshf} satisfy the hypothesis of Theorem \ref{thmlapshf}. 
Among them, the example on $\fre(1,1)\oplus\fre(1,1)$ is the only one satisfying the hypothesis of Corollary \ref{corproductthm}. 
\end{proposition}
\begin{proof}
It is sufficient to describe the expression of the inner product $h$ and the intrinsic torsion form $\Wd$. The properties can then be checked by straightforward computations. 
\[
\renewcommand\arraystretch{1.5}
\begin{array}{rl}
\fre(1,1)\oplus\fre(1,1):			& h =\sum_{i=1}^4(e^i)^2 + 2\sum_{i=5}^6(e^i)^2,	\quad \Wd =2\,\left(e^{26}+e^{25}+e^{36}-e^{35}\right);\\
\frg_{5,1}\oplus\R:				& h =\sum_{i=1}^6(e^i)^2,						\quad \Wd = e^{26}-e^{35};\\
A_{5,7}^{-1,-1,1}\oplus\R:			& h =\sum_{i=1}^6(e^i)^2,						\quad \Wd =2\left(e^{14}-e^{23}\right);\\
A_{5,17}^{\alpha,-\alpha,1}\oplus\R:	& h =\sum_{i=1}^6(e^i)^2,						\quad \Wd =-2\alpha\left(e^{12}+e^{34}\right);\\
\frg_{6,N3}:					& h =\sum_{i=1}^5(e^i)^2+4(e^6)^2,				\quad \Wd = 4\,e^{16}-e^{25}+e^{34};\\
\frg_{6,38}^{0}:					& h =4(e^1)^2 + \sum_{i=2}^6(e^i)^2,				\quad \Wd =4\,e^{16}-e^{25}+e^{34};\\
\frg_{6,54}^{0,-1}:				& h =\sum_{i=1}^5(e^i)^2+2(e^6)^2,				\quad \Wd = \sqrt{2}\,e^{13}-e^{14}+e^{23}+\sqrt{2}\,e^{24};\\
\frg_{6,118}^{0,-1,-1}:			& h =\sum_{i=1}^6(e^i)^2,						\quad \Wd =2\left(e^{12}-e^{34}\right).
\end{array}
\]
\end{proof}

\begin{remark}
The solution of the Laplacian flow on the seven-dimensional nilpotent Lie algebra $\frn_2\coloneqq \frg_{5,1}\oplus\R\oplus\R$ coming from the above result 
was previously obtained in \cite[Thm.~4.2]{FFM} using a different method.  
\end{remark}

\begin{table}[ht]
\centering
\renewcommand\arraystretch{1.4}
\adjustbox{max width=\textwidth}{
\begin{tabular}{|c|c|c|c|}
\hline
Lie algebra				&symplectic half-flat $\SU(3)$-structure							&$\Delta_h\Wd=c\,\Wd$	\\ \hline \hline
$\fre(1,1)\oplus\fre(1,1)$				&	$\begin{array}{c} \omega =e^{14}+e^{23}+2e^{56} \\ 
									\psip =e^{125}-e^{126}-e^{135}-e^{136}+e^{245}+e^{246}+e^{345}-e^{346}  \end{array}$			&	$c=2$			\\ \hline    

$\frg_{5,1}\oplus\R$					&	$\begin{array}{c} \omega = e^{14} +e^{26}+e^{35}\\ 
									\psip = e^{123}+e^{156}+e^{245}-e^{346} \end{array} $									&	$c=2$			\\ \hline

$A_{5,7}^{-1,-1,1}\oplus\R$			&	$\begin{array}{c} \omega = -e^{13}+e^{24}+e^{56}
									\\ \psip =  -e^{126}-e^{145}-e^{235}-e^{346}\end{array}$									&	$c=4$			\\ \hline

$A_{5,17}^{\alpha,-\alpha,1}\oplus\R$	&	$\begin{array}{c} \omega  = e^{13}+e^{24}+e^{56}
									\\ \psip = e^{125}-e^{146}+e^{236}-e^{345} \end{array} $									&	$c=4\alpha^2$		\\ \hline

$\frg_{6,N3}$						&	$\begin{array}{c} \omega = 2e^{16}+e^{25}-e^{34} 
									\\ \psip = -e^{123}+e^{145}-2e^{246}-2e^{356} \end{array} $								&	$c=6$			\\ \hline

$\frg_{6,38}^{0}	$					&	$\begin{array}{c} \omega = -2e^{16}+e^{34}-e^{25}
									\\ \psip = -2e^{135}-2e^{124}+e^{236}-e^{456} \end{array} $								&	$c=6$			\\ \hline

$\frg_{6,54}^{0,-1}$					&	$\begin{array}{c} \omega =e^{14}+e^{23}+\sqrt{2}e^{56} \\ 
									\psip =e^{125}-\sqrt{2}e^{136}+\sqrt{2}e^{246}+e^{345}  \end{array} $						&	$c=2$			\\ \hline

$\frg_{6,118}^{0,-1,-1}$				&	$\begin{array}{c} \omega = e^{14}+e^{23}-e^{56} \\ 
									\psip = e^{126}-e^{135}+e^{245}+e^{346} \end{array} $									&	$c=4$			\\ \hline

\end{tabular}}
\vspace{0.1cm}
\caption{Examples of symplectic half-flat $\SU(3)$-structures on the Lie algebras listed in Theorem \ref{solvableSHF}.}\label{tabshf}
\end{table}
\renewcommand\arraystretch{1}

A further question arising in this setting is whether the left-invariant closed $\G_2$-structures considered above 
define a Laplacian soliton on  $\hat\G$.  
By \cite[Thm.~4.10]{Lau1}, a sufficient condition is given by the existence of a derivation $D\in{\rm Der}(\hat\frg)$ and a real number $\lambda$ such that 
\begin{equation}\label{LapSolDer}
\Delta_\f\f = \mathcal{L}_{X_D}\f+\lambda\,\f,
\end{equation}
where $X_D$ is the vector field on $\hat\G$ induced by the one-parameter group of automorphisms with derivative $\exp(tD)\in{\rm Aut}(\hat\frg)$. 

\begin{proposition}
Let $\frg$ be one of the unimodular solvable Lie algebras appearing in Theorem \ref{solvableSHF}, and endow it with the   
symplectic half-flat $\SU(3)$-structure $(\omega,\psip)$ of Table \ref{tabshf}. 
Then, the closed $\G_2$-structure $\f=\omega\W e^7+\psip$ on $\hat\frg=\frg\oplus\R$ gives rise to a Laplacian soliton on $\hat\G$ 
when $\frg$ is $\fre(1,1)\oplus\fre(1,1)$, $\frg_{5,1}\oplus\R$, $A_{5,7}^{-1,-1,1}\oplus\R$, $A_{5,17}^{\alpha,-\alpha,1}\oplus\R$, 
$\frg_{6,N3}$, $\frg_{6,54}^{0,-1}$. 
\end{proposition}
\begin{proof}
It is sufficient to check that the closed $\G_2$-structure $\f=\omega\W e^7+\psip$ satisfies \eqref{LapSolDer} for a suitable derivation $D\in{\rm Der}(\hat\frg)$ 
and a real number $\lambda$. We have:
\[
\begin{array}{rl}
\fre(1,1)\oplus\fre(1,1):			& D=\diag\left(0,0,0,0,0,0,-2\right),			\quad	 \lambda=2;\\
\frg_{5,1}\oplus\R:				& D=\diag\left(-1,-1,-1,-2,-2,-2,-2\right),		\quad	 \lambda=5;\\
A_{5,7}^{-1,-1,1}\oplus\R:			& D=\diag\left(-2,-2,-2,-2,0,-4,-4\right),		\quad	 \lambda=8;\\
A_{5,17}^{\alpha,-\alpha,1}\oplus\R:	& D=\alpha^2\,\diag\left(-2,-2,-2,-2,0,-4,-4\right),	\quad 	 \lambda=8\alpha^2;\\
\frg_{6,N3}:					& D=\diag\left(-3,-3,-3,-6,-6,-6,-6\right),		\quad	 \lambda=15;\\
\frg_{6,54}^{0,-1}:				& D=\diag\left(-1,-1,0,0,-1,0,-2\right),			\quad	 \lambda=3.
\end{array}
\]
Notice that in each case $2 \lambda = 6c-|\Wd|_h^2$.
\end{proof}

\begin{remark}
Since the derivations appearing in the proof of the previous theorem are all symmetric, the closed $\G_2$-structures on $\hat\frg$ are all {\em algebraic solitons} in the sense of 
\cite{Lau1,Lau2}. Observe that the closed algebraic solitons on the nilpotent Lie algebras $\frn_2\coloneqq \frg_{5,1}\oplus\R\oplus\R$ and $\frn_3\coloneqq \frg_{6,N3}\oplus\R$ 
were already obtained in \cite{FFM,Lau1,Nic}. 
Moreover, even if the Lie group $\hat\G$ admits a lattice $\Gamma$, the vector field $X_D$ does not descend to the compact solvmanifold $\Gamma\backslash\hat\G$. 
Hence, the invariant closed $\G_2$-structures on $\Gamma\backslash\hat\G$ arising in this context are not necessarily Laplacian solitons. 
Finally, a straightforward computation shows that the two remaining examples of closed $\G_2$-structures on $\frg_{6,38}^{0}\oplus\R$ and $\frg_{6,118}^{0,-1,-1}\oplus\R$ are not 
algebraic solitons.
\end{remark}

In the next example, we study the behaviour of the Laplacian flow when the initial datum is a symplectic half-flat $\SU(3)$-structure whose intrinsic torsion form $\Wd$ is not 
an eigenform of the Hodge Laplacian.  In particular, we obtain that the metric $g_{\f(t)}$ induced by the solution is a Riemannian product for each $t$.
\begin{example}\label{exWdnoteigenform}
Consider the Lie algebra $\frg_{6,54}^{0,-1}$ endowed with the symplectic half-flat $\SU(3)$-structure described in \cite[Ex.~3.1]{TomVez}
\[
\omega = e^{14}+e^{23}+e^{56},\qquad \psip = e^{125}-e^{136}+e^{246}+e^{345}.
\]
The inner product induced by $(\omega,\psip)$ is $h=\sum_{i=1}^6(e^i)^2$. 
A simple computation shows that $\Wd$ is not an eigenform of the Hodge Laplacian $\Delta_h$. Indeed
\[
\Wd = e^{23}-e^{14}+2\,e^{13}+2\,e^{24},\quad \Delta_h\Wd = 2\,e^{23}-2\,e^{14}+8\,e^{13}+8\,e^{24}.
\] 
Thus, we cannot apply Theorem \ref{thmlapshf} to obtain the solution of the Laplacian flow on $\frg_{6,54}^{0,-1}\oplus\R$ starting from the closed $\G_2$-structure
\[
\f=\omega\W e^7+\psip = e^{147}+e^{237}+e^{567}+e^{125}-e^{136}+e^{246}+e^{345}.
\]
Anyway, we can find it following the approach used in the proof of \cite[Thm.~4.2]{FFM}. 
The idea is to consider a new basis of covectors ${\tilde{e}}^k\coloneqq v_k\,e^k$, $k=1,\ldots,7,$ 
where $v_k=v_k(t)$ are nowhere vanishing real-valued functions satisfying $v_k(0)=1$, 
and look for a solution of the following type
\[
\f(t)	=  v_{147}\,e^{147}+v_{237}\,e^{237}+v_{567}\,e^{567}+v_{125}\,e^{125}-v_{136}\,e^{136}+v_{246}\,e^{246}+v_{345}\,e^{345},
\]
where the notation $v_{ijk}$ is used for the product of functions $v_iv_jv_k$. 
A straightforward computation gives
\[
\f(t)=e^{147}+e^{237}+e^{567}+e^{125}-v_1^4v_4^2\,e^{136}+v_1^4v_4^2\,e^{246}+\left(\frac{v_4}{v_1}\right)^2e^{345},
\]
where $v_1$ and $v_4$ are the solutions of the following initial value problem
\[
\left\{
\renewcommand\arraystretch{1.5}
\begin{array}{l}
\frac{{\rm d}v_1}{{\rm d}t} = \frac{2-v_1^{12}}{3\,v_1^5v_4^2},\\
\frac{{\rm d}v_4}{{\rm d}t} = \frac23\frac{1+v_1^{12}}{v_1^6v_4},\\
v_1(0) = v_4(0) =1.
\end{array}
\right.
\]
Observe that $\f(t)$ can be written as 
\[
\f(t) = f(t)\, \omega(t)\W e^7 +\psip(t),
\]
where 
\[
\omega(t) = v_1v_4\left(e^{14}+e^{23}+e^{56}\right),\quad \psip(t) =e^{125}-v_1^4v_4^2\,e^{136}+v_1^4v_4^2\,e^{246}+\left(\frac{v_4}{v_1}\right)^2e^{345},
\]
is a family of symplectic half-flat $\SU(3)$-structures on  $\frg_{6,54}^{0,-1}$, and $f(t) = \left(v_1(t)v_4(t)\right)^{-1}.$ 
Hence, $g_{\f(t)} = h(t) + [f(t)]^2(e^7)^2$, where $h(t)$ is the metric induced by $(\omega(t),\psip(t))$. 
\end{example}

\subsection{Twistor spaces} 
Let $\left(M^4,g_4,dV_4\right)$ be an oriented Riemannian 4-manifold, and denote by $\mathcal{Z}$ its twistor space, namely the set of all almost complex structures on $M^4$ which are 
compatible with $g_4$ and preserve the orientation. 
$\mathcal{Z}$ is a six-dimensional manifold, and it can be naturally endowed with a family of $\SU(3)$-structures (see for instance \cite{EeSa}).

Assume that $\left(M^4,g_4\right)$ is self-dual Einstein of negative scalar curvature. 
Up to scale $g_4$, we can suppose that ${\rm Scal}(g_4) = -48$. Then, $\mathcal{Z}$ is endowed with a symplectic half-flat 
$\SU(3)$-structure $(\omega,\psip)$ satisfying
\[
d_6\Wd =\frac{\left|\Wd\right|_h^2}{4} \psip,
\]
where $\left|\Wd\right|_h^2=96$ (see e.g.~\cite[Sect.~1.2]{Xu} for explicit computations). 
Compact examples of this type can be obtained when $M^4$ is the compact quotient of the four-dimensional hyperbolic space by some discrete isometry group.

By Corollary \ref{corproductthm}, the solution of the system \eqref{systcompletefomgpsip} with initial conditions $f(0)=1$ and $(\omega(0),\psip(0))=(\omega,\psip)$ is
\[
f(t) = \left(16t+1\right)^{-1},\quad \omega(t) = \left(16t+1\right)\,\omega,\quad \psip(t)=\left(16t+1\right)^{\frac32}\psip.
\]
Therefore, the Laplacian flow on $\mathcal{Z}\times\Su$ starting from the closed $\G_2$-structure $\f=\omega\W ds+\psip$ has the following solution 
\[
\f(t) = \omega \W ds + \left(16t+1\right)^{\frac32}\psip,
\]
and the associated Riemannian metric is $g_{\f(t)} = \left(16t+1\right)h + \left(16t+1\right)^{-2}ds^2.$

\subsection{Non-flat closed G$_{\mathbf 2}$-structures on the 7-torus} 
In Example \ref{extorus}, we considered two different families of $\SU(3)$-structures on the 6-torus $\T^6$.  
Here, we use them to construct two examples of non-flat closed $\G_2$-structures on $\T^7$. 

Let us begin with the family of symplectic half-flat $\SU(3)$-structures introduced in \cite{TomVez}. 
It is defined for each $r>0$ by the pair $(\omega,\psip(r))$ given in \eqref{SHFT6}, 
and it induces the following Riemannian metric
\[
h(r) = e^{-r\,\ell_1}(dx^1)^2 + {e^{-r\,\ell_2}}(dx^2)^2 + {e^{-r\,\ell_3}}(dx^3)^2 + {e^{r\,\ell_1}}(dx^4)^2 + {e^{r\,\ell_2}}(dx^5)^2 + {e^{r\,\ell_3}}(dx^6)^2.
\]
Using identity \eqref{ScalSHF}, it is not difficult to check that the scalar curvature of $h(r)$ is nonconstant:
\[
{\rm Scal}(h(r)) = 	-r^2 \left( \left( \frac{\rm d}{{\rm d}x^3} c(x^3)  \right)^2 {e}^{r\ell_3} 
				+ \left( \frac{\rm d}{{\rm d}x^1} a(x^1)  \right)^2 {e}^{r\ell_1}    
				+ \left( \frac{\rm d}{{\rm d}x^2} b(x^2)  \right)^2 {e}^{r\ell_2}   \right).
\]
As observed in \cite[Prop.~1.4.4]{ManTh}, the 7-torus $\T^7=\T^6\times\Su$ is endowed with the family of closed $\G_2$-structures
$\f(r) \coloneqq \omega\W ds +\psip(r)$, whose associated Riemannian metric $g_{\f(r)} = h(r) + ds^2$ is non-flat.

The second family of $\SU(3)$-structures on $\T^6$ appearing in Example \ref{extorus} is defined by the 3-form $\psip(r)$ and the 2-form 
$\widehat\omega = \frac{1}{\kappa_1\kappa_2\kappa_3}\widetilde\omega$, where $\widetilde\omega$ is given by \eqref{tildeomgT6}. 
Since $(\widehat\omega,\psip(r))$ satisfies \eqref{structueqnsu3} with $\theta =  -d_6 \log\left(\kappa_1\kappa_2\kappa_3\right)$, 
by Lemma \ref{G2closedSU3corresp} we know that the 3-form
\[
\widehat\f(r)\coloneqq  \left(\kappa_1\kappa_2\kappa_3\right)\widehat\omega\W ds +\psip(r)
\]
defines a closed warped $\G_2$-structure on $\T^7=\T^6\times\Su$. Observe that the Riemannian metric induced by $\widehat\f(r)$ is
\[
g_{\widehat\f(r)} = \widehat{h}(r) +  \left(\kappa_1\kappa_2\kappa_3\right)^2 ds^2,
\]
$\widehat{h}(r)$ being the metric associated with the $\SU(3)$-structure $(\widehat\omega,\psip(r))$. 

Having provided two examples of closed $\G_2$-structures on $\T^7$, 
a challenging problem consists in studying the behaviour of the Laplacian flow \eqref{LapFlow} starting from one of them.  

\bigskip

\noindent  {\bf Acknowledgements.} The authors would like to thank Fabio Podest\`a and Luigi Vezzoni for useful conversations, and Jorge Lauret for useful comments.


\end{document}